\documentclass[11pt,a4paper,reqno]{amsart}

\pdfoutput=1

\usepackage[utf8]{inputenc}
\usepackage[british]{babel}
\usepackage{amsmath,amssymb,amsthm}
\usepackage{thmtools,thm-restate}
\usepackage{hyperref,float,caption}
\hypersetup{
    colorlinks,
    linkcolor={red!60!black},
    citecolor={green!60!black},
    urlcolor={blue!60!black},
}
\usepackage[nameinlink]{cleveref}
\usepackage[margin=1in]{geometry}
\usepackage{enumitem}
\usepackage[abbrev, msc-links, nobysame]{amsrefs}
\usepackage{tikz}
\usepackage{xcolor}
\usepackage{accents}
\usepackage[presets={vec-cev}]{letterswitharrows}

\usepackage{subcaption}

\usetikzlibrary{calc}
\usetikzlibrary{arrows}
\tikzset{
e+ /.tip = {[sep=0pt]|[sep=3pt]_},
e+4 /.tip = {[sep=0pt]|[sep=5pt]_},
e+5 /.tip = {[sep=0pt]|[sep=5pt]_},
e+6 /.tip = {[sep=0pt]|[sep=6pt]_},
e+8 /.tip = {[sep=0pt]|[sep=9pt]_},
e+2 /.tip = {[sep=0pt]|[sep=2pt]_}
}
\tikzset{arc/.style = {->,> = stealth', shorten >=2.3pt}}
\tikzset{
dot/.style = {circle, draw = black, fill, minimum size=#1,
              inner sep=0pt, outer sep=0pt},
dot/.default = 3.8pt
}
\usetikzlibrary{angles,quotes}
\definecolor{CornflowerBlue}{rgb}{0.39, 0.58, 0.93}
\definecolor{LavenderMagenta}{rgb}{0.93, 0.51, 0.93}
\definecolor{PastelOrange}{rgb}{1.0, 0.7, 0.28}

\def\Y{\mathcal Y}

\def\halfedges{\textbf{E}}

\newcommand{\skeleton}[1]{D_{#1}^{\mathrm{path}}}
\newcommand{\trailskeleton}[1]{D_{#1}^{\mathrm{trail}}}

\newcommand{\set}[1]{\left\lbrace #1 \right\rbrace}

\newcommand{\COMMENT}[1]{} 

\DeclareMathOperator{\InVertices}{In}

\newtheorem{theorem}{Theorem}[section]
\newtheorem{lemma}[theorem]{Lemma}
\newtheorem{corollary}[theorem]{Corollary}

\newtheorem{claim}{Claim}[theorem]
\theoremstyle{definition}

\newtheorem{remark}[theorem]{Remark}
\newtheorem*{theorem*}{Theorem}

\crefname{enumi}{}{}
\crefname{enumii}{}{}
\crefformat{enumi}{#2#1#3}
\crefformat{enumii}{#2#1#3}
\Crefformat{enumi}{#2#1#3}
\Crefformat{enumii}{#2#1#3}

\crefname{definition}{definition}{definitions}
\crefformat{definition}{#2Definition~#1#3}
\Crefformat{definition}{#2Definition~#1#3}

\crefname{section}{section}{sections}
\crefformat{section}{#2Section~#1#3}
\Crefformat{section}{#2Section~#1#3}

\crefname{subsection}{Subsection}{subsections}
\crefformat{subsection}{#2Subsection~#1#3}
\Crefformat{subsection}{#2Subsection~#1#3}

\crefname{lemma}{lemma}{lemmata}
\crefformat{lemma}{#2Lemma~#1#3}
\Crefformat{lemma}{#2Lemma~#1#3}

\crefname{remark}{remark}{remarks}
\crefformat{remark}{#2Remark~#1#3}
\Crefformat{remark}{#2Remark~#1#3}

\crefname{theorem}{theorem}{theorems}
\crefformat{theorem}{#2Theorem~#1#3}
\Crefformat{theorem}{#2Theorem~#1#3}

\crefname{corollary}{corollary}{corollaries}
\crefformat{corollary}{#2Corollary~#1#3}
\Crefformat{corollary}{#2Corollary~#1#3}

\crefname{figure}{figure}{figures}
\crefformat{figure}{#2Figure~#1#3}
\Crefformat{figure}{#2Figure~#1#3}

\crefname{proposition}{proposition}{propositions}
\crefformat{proposition}{#2Proposition~#1#3}
\Crefformat{proposition}{#2Proposition~#1#3}

\crefname{observation}{observation}{observations}
\crefformat{observation}{#2Observation~#1#3}
\Crefformat{observation}{#2Observation~#1#3}

\crefname{claim}{claim}{claim}
\crefformat{claim}{#2Claim~#1#3}
\Crefformat{claim}{#2Claim~#1#3}

\def\lqedsymbol{\ifmmode$\lrcorner$\else{\unskip\nobreak\hfil
		\penalty50\hskip1em\null\nobreak\hfil$\rule{1.2ex}{1.2ex}$
		\parfillskip=0pt\finalhyphendemerits=0\endgraf}\fi}

\newenvironment{claimproof}[1][\proofname]
{%
	\proof[#1]%
}
{%
	\endproof%
}

\title{A structure theorem for rooted connectivity in bidirected graphs}
\author[T.\,Abrishami]{Tara Abrishami}
\author[N.\,Bowler]{Nathan Bowler}
\author[A.\,Jo\'{o}]{Attila Jo\'{o}}
\author[F.\,Reich]{Florian Reich}
\author[Q.\,Tao]{Qiuzhenyu Tao}
\address{Tara Abrishami: Stanford University, Department of Mathematics}
\email{tara.abrishami@stanford.edu}

\address{Nathan Bowler, Attila Jo\'{o}, Florian Reich, Qiuzhenyu Tao: Universit{\"a}t Hamburg, Department of Mathematics, Bundesstra{\ss}e~55 (Geomatikum), 20146~Hamburg, Germany}
\email{\{nathan.bowler, attila.joo, florian.reich, qiuzhenyu.tao\}@uni-hamburg.de}

\thanks{T.A. is supported by the National Science Foundation Award Number DMS-2303251 and the Alexander von Humboldt Foundation.
A.J. was funded by the Deutsche Forschungsgemeinschaft (DFG) - Grant No. 513023562.
F.R. is supported by a doctoral scholarship of the Studienstiftung des deutschen Volkes.
Q.T. is supported by China Scholarship Council (No.~202206770031) and Sino-German (CSC-DAAD) Postdoc Scholarship Program, 2022 (57607866).
}

\begin{document}

    

\begin{abstract}
    Recently, bidirected graphs have received increasing attention from the graph theory community with both structural and algorithmic results.
   Bidirected graphs are a generalization of directed graphs, consisting of an undirected graph together with a map assigning each endpoint of every edge either sign $+$ or $-$. The connectivity properties of bidirected graphs are more complex than those of directed graphs and not yet well understood.
    
    In this paper, we show a structure theorem about rooted connectivity in bidirected graphs in terms of directed graphs.
%
    As applications, we prove Lovász's flame theorem, Pym's theorem and a strong variant of Menger's theorem for a class of bidirected graphs and provide counterexamples in the general case.
    
\end{abstract}
\maketitle

\section{Introduction}

A \emph{bidirected graph} is a graph $B$ endowed with a sign function $\sigma$ that assigns a sign $\{+, -\}$ to each end of every edge of $B$. The purpose of the sign function is to restrict the structure of walks in the bidirected graph: if two edges $uv$ and $vw$ appear consecutively in a walk, then the sign of edge $uv$ at $v$ must be the opposite of the sign of edge $vw$ at $v$. Stated informally: every walk must change signs when passing through a vertex.

Bidirected graphs were introduced independently by Kotzig \cites{Kotzig1959, Kotzig1959b, Kotzig1960} and by Edmonds and Johnson \cite{edmonds2003matching}. Recently, much research on bidirected graphs has been motivated by the fact that they are related to the structure of perfect matchings in (undirected) graphs \cite{wiederrecht2020thesis, bowler2023mengertheorembidirected, bowler2024hitting, bowler2025ConnectoidsI, KITA2017, Nickel2025, Ghorbani2025}. In particular, there is a way to translate between undirected graphs with perfect matchings and bidirected graphs. Suppose that $G$ is an undirected graph with a perfect matching $M$. We construct a bidirected graph $(B, \sigma)$ from $G$ as follows. For each edge $uv$ of $M$: 
\begin{itemize} 
\item choose an end $u$ arbitrarily;
\item assign every edge $uw$ with $w \neq v$ the sign $-$ at $u$; 
\item assign every other edge $vw$ with $w \neq u$ the sign $+$ at $v$; and 
\item contract edge $uv$. 

\end{itemize}
Because $M$ is a perfect matching, this procedure assigns every end of every edge a sign. This procedure is reversible: from every bidirected graph, we can construct a corresponding undirected graph with a perfect matching. 


There are several open problems and research directions in structural graph theory dedicated to studying the structure of perfect matchings in graphs. Matching minors and perfect matching width were introduced as a substructure and a width parameter that interact well with the structure of matchings in graphs. Norin \cite{norine2005matching} proposed the Matching Grid Conjecture, an analog of the Grid Minor Theorem, to relate these two concepts. The conjecture states that, given a graph with a perfect matching, either the graph has bounded perfect matching width or it contains a large grid as a matching minor. Using the correspondence between graphs with perfect matchings and bidirected graphs, one can show that matching minors are a bidirected graph analog of butterfly minors (which are themselves a directed graph analog of minors). In fact, by translating matching minors, perfect matching width, and the Matching Grid Conjecture to the bidirected graph setting, the Matching Grid Conjecture can be understood as a Bidirected Grid Conjecture, i.e. a conjectured grid theorem for bidirected graphs. 

Recently, Hatzel, Rabinovich, and Wiederrecht \cite{hatzel2019cyclewidth} proved the Matching Grid Conjecture for bipartite graphs. Their proof relied on two important insights. First, if we start with a bipartite graph $G$ with a perfect matching and construct a bidirected graph $(B, \sigma)$ from $G$ using the procedure described above, the bidirected graph $(B, \sigma)$ is in fact a directed graph, i.e. every edge in $(B, \sigma)$ can be traversed in at most one direction. They then used the Directed Grid Theorem \cite{kawarabayashi2015directed} to prove the Matching Grid Conjecture in this case. 

Inspired by the importance of understanding connectivity to the Matching Grid Conjecture, we study the connectivity of bidirected graphs. Because of the way that walks are defined in bidirected graphs, trails and paths have more complicated structures than in undirected or directed graphs. For example, if a bidirected graph $B$ contains a trail from $x$ to $y$, it may not necessarily contain a path from $x$ to $y$. Since connectivity in bidirected graphs is so complex, we make two assumptions: we work with rooted graphs (and rooted connectivity), and we assume further that the bidirected graphs are \emph{clean}, meaning that they exclude certain types of trails. 

The main result of this paper is a structure theorem for rooted clean bidirected graphs that models any given bidirected graph on a directed graph. The rooted connectivity of the bidirected graph can then be determined through the rooted connectivity of the digraph model. We in fact have two versions of this structure theorem, one tailored to studying edge connectivity and one tailored to vertex connectivity. Using these structure theorems, we prove versions of Lovasz's flame theorem, Pym's theorem, and Menger's theorem for edge-disjoint and vertex-disjoint paths in bidirected graphs. We also give examples showing that these results may fail without our assumption of cleanness.

\subsection{Overview of the proof}
The goal of our main result is to model bidirected graphs on directed graphs that represent their connectivity. Let $B$ be a rooted clean bidirected graph. First, we partition the edges of $B$ into two groups: the directable edges, which can be traversed in at most one direction on paths (resp. trails) from the root, and the undirectable edges, which can be traversed in both directions in this way. (If every edge of $B$ is directable, then $B$ is in fact  a directed graph.) The set of directable edges forms the basis for the directed graph on which we model $B$. The undirectable edges, meanwhile, can be partitioned into connected components of vertices that each have connectivity one from the root of $B$. These components are represented in the directed graph via a single vertex.

All of our results have algorithmic versions. The main difficulty is to find an algorithm which checks in polynomial time whether an edge is trail-directable. However, as explained in \cite{bowler2023mengertheorembidirected}, such an algorithm can be quickly obtained from Edmonds’ Blossom algorithm. Given this, it should be clear from our constructions that each of our decompositions of a bidirected graph can also be constructed in polynomial time, which yields efficient algorithms for all our applications.

\subsection{Structure of the paper} We begin by introducing the notation related to bidirected graphs in~\cref{sec: notation}. In \cref{sec: edge decmp}, we prove our main result, a structure theorem for rooted connectivities of bidirected graphs. In \Cref{sec:edge-menger,sec:edge-flame,sec:edge-pym}, we use our structure theorem to generalize the edge-disjoint versions of Menger's theorem, Lovász flame theorem, and Pym's theorem, respectively, to edge-clean bidirected graphs. In \cref{sec: from edge to vertex}, we derive the vertex-disjoint variants of these applications for bidirected graphs. Finally, we adapt our structure theorem to the vertex-disjoint context in \cref{sec: vertex-decomp}.

\section{Preliminaries and Notation}\label{sec: notation}
For the standard notation for undirected graphs we refer the reader to \cite{diestel2024graph}.
 A \emph{bidirected graph}~$B = (G, \sigma)$ consists of a loopless multigraph~$G = (V, E)$, a corresponding set of \emph{half-edges} defined as
    \begin{equation*}
        \halfedges(B) := \set{ (e,u) \colon e \in E \text{ and } u \text{ is an endpoint of }  e }
    \end{equation*}
and a \emph{signing}~$\sigma: \halfedges \to \set{+, -}$ assigning to each half-edge~$(e,u)$ its \emph{sign}~$\sigma(e,u) := \sigma((e,u))$; we say that~\emph{$e$ has sign~$\sigma(e,u)$ at~$u$}.
Then~$V(B) := V$ is the \emph{vertex set} of~$B$ and~$E(B) := E$ is its \emph{edge set}.
We refer to the elements of~$V(B)$ and~$E(B)$ as the \emph{vertices} and the \emph{edges} of~$B$, respectively. A \emph{signed vertex} of~$B$ is a pair~$(v, \alpha)$ of a vertex~$v$ of~$B$ and a sign~$\alpha \in \{+, -\}$.
Given a vertex set~$X$, we let~$\mathcal{V}(X) := \{ (v,\alpha):\  v\in V(B), \alpha  \in \{+, -\}\}$. For an edge set $F$, we write $V(F)$ for the set of the endpoints of the edges in $F$.

An \emph{oriented edge}~$\ve$ of a bidirected graph~$B$ is formally defined as a triple~$(e, u, v)$ where~$e$ is an edge of~$B$ with incident vertices~$u, v \in V(B)$; we call~$e$ its \emph{underlying edge},~$u$ its \emph{tail} and~$v$ its~\emph{head}, and think of~$\ve$ as orienting~$e$ from~$u$ to~$v$.
The edge~$e$ has precisely two \emph{orientations}, one with startvertex~$u$ and endvertex~$v$ and the other one with startvertex~$v$ and endvertex~$u$.
We denote the two orientations of~$e$ as~$\ve$ and~$\ev$; there is no default orientation of~$e$, but if we are given one of them as~$\ve$, say, then the other one is written as~$\ev$.
Given an oriented edge~$\ve$ we conversely write~$e$ for its underlying edge.

In every drawing involving bidirected graphs, we depict the signs of halfedges by drawing the signs onto the edge, which results in a bar perpendicular to the edge at incident vertices with sign $+$ and none at incident vertices with sign $-$, see e.g.\ on the left at  \cref{fig:edge-decomp}.

A \emph{trail} $T$ in $B$ is a sequence $v_0 \ve_1 v_1 \ve_2 v_2 \dots v_{\ell - 1} \ve_\ell v_\ell$ of vertices $(v_i)_{0 \leq i \leq \ell}$ and oriented edges $(\ve_i)_{1 \leq i \leq \ell}$ of $B$ where $e_i\neq e_j$ for $i\neq j$ such that $\ve_{i}$ has tail $v_{i-1}$ and head $v_i$ for $i \in [\ell]$ and such that $\sigma(e_{i+1}, v_i) = - \sigma(e_i, v_i)$ for $i \in [\ell - 1]$. Trails consisting of a single vertex are called \emph{trivial}. Let $E(T)$ be the set of edges that have an orientation appearing in $T$. 
Let $T^{-1}$ stand for the trail that we obtain by going along $T$ backwards, more precisely by 
reversing the orientations and the order of the vertices and edges of $T$. The trails $S$ and $T$ are \emph{edge-disjoint} if $E(S)\cap E(T)=\emptyset$. If $\ve$ is the last directed edge of a trail $T$, then we denote by $T - \ve$ the trail obtained by omitting the last edge and vertex of $T$. If adding a trail $T$ at the end of a trail $S$ again yields a trail, then it is denoted by $S \circ T$.  

An $x$--$y$~\emph{trail} is a trail with initial vertex $x$ and terminal vertex $y$, and an $x$--$\vf$~\emph{trail} is a trail with initial vertex $x$ and terminal edge $\vf$. For $\beta \in \{+,-\}$, an $x$--$(y,\beta)$~trail is an $x$--$y$~trail that arrives at $y$ with sign $\alpha$. Similarly, we define $\ve$--$y$, $\ve$--$\vf$, $\ve$--$(y,\beta)$~trails and $(x,\alpha)$--$y$, $(x, \alpha)$--$\vf$, $(x, \alpha)$--$(y,\beta)$~trails.

A bidirected graph $B$ is called \emph{rooted} if there is a designated root vertex that we denote by $r$. An $r$-\emph{trail} is a trail that starts in $r$ but with no internal vertex equal to $r$. So such a trail never comes back to $r$, except possibly at its final vertex. An $r$--$\ve$~\emph{trail} is an  $r$-trail whose terminal edge is $\ve $, while an $r$--$e$~\emph{trail} is an  $r$--$\ve$~trail for either orientation $\ve$ of $e$.
A \emph{path} is a trail that does not repeat vertices.
An {\em almost path} is a trail that is either trivial or such that the removal of its last edge results in a path. In a rooted bidirected graph edge $e$ is \emph{(path/trail)-reachable} if there exists an $r$--$e$~\emph{(path/trail)}. A rooted bidirected graph is \emph{(path/trail)-reachable} if every edge is (path/trail)-reachable.

%
\section{The edge-decomposition}\label{sec: edge decmp}




An edge $e \in E(B)$ is {\em trail-undirectable} if $B$ contains both an $r$--$\ve$~trail and an $r$--$\ev$~trail. An edge $e \in E(B)$ is {\em trail-directable} if there is a unique orientation $\ve$ of $e$ such that $B$ contains an $r$--$\ve$~trail; we call this orientation $\ve$ the {\em natural orientation} of $e$. 

Let $B$ be a bidirected graph rooted at $r$.
A \emph{component} of a bidirected graph is maximal subgraph whose underlying undirected graph is connected. We say a component is \emph{nontrivial} if it contains at least one edge.
Let $C_1, \hdots, C_k$ be the nontrivial components of the subgraph of $B$ induced by the set of trail-undirectable edges, called the {\em trail-undirectable components} of $B$.

\begin{lemma}\label{lem:component_properties}
    Let $B$ be a bidirected graph rooted at $r$ and let $C_1, \hdots, C_k$ be its nontrivial trail-undirectable components. For every such component $C_i$ with $r \not \in V(C_i)$, the following holds: 
    \begin{enumerate}[label=(\alph*)]
        \item\label{itm:3.1i} there is exactly one trail-directable edge $f_i$ such that $\vf_i$ has head in $C_i$;
        \item\label{itm:3.1ii} every edge $e$ of $B$ with both endpoints in $C_i$ is trail-undirectable and there is an $\vf_i$--$\ve$~trail and an $\vf_i$--$\ev$~trail with all internal vertices in $C_i$; and
        \item\label{itm:3.1iii} for every vertex $w \in C_i$, there is an $\vf_i$--$(w,-)$~trail $T^{(w,-)}$ and an $\vf_i$--$(w,+)$~trail $T^{(w,+)}$ in $B$ with all edges other than $f_i$ in $E(C_i)$.
    \end{enumerate}
\end{lemma}
\begin{proof}
    We know that $r \not \in V(C_i)$, $C_i$ spans at least one (trail-undirectable) edge $g$ by definition and there exists an $r$--$g$~trail. Thus we conclude that there is a nontrivial trail $S$ from $r$ to $C_i$ that first meets $C_i$ only at its last vertex. Let $\vf_i$ be the last edge of $S$. First we show that $f_i$ satisfies conditions \ref{itm:3.1ii} and \ref{itm:3.1iii}. Let $u$ and $v$ be the tail and head of $\vf_i$, respectively. Clearly, $f_i$ is trail-directable since otherwise we would have $u \in V(C_i)$.

    Let $F \subseteq E(B)$ be maximal with respect to the property that for every $e \in F$, there is an $\vf_i$--$\ve$~trail and an $\vf_i$--$\ev$~trail with all edges in $F \cup \{f_i\}$. 
    
    \begin{claim}\label{F-nonempty} $v\in V(F) $.  \end{claim} 
    \begin{claimproof} First we claim that there is a trail-undirectable edge $e$ incident with $v$ such that $\sigma(f_i, v) = -\sigma(e, v)$. Let $h \in E(C_i)$ be a trail-undirectable edge incident with $v$ (the existence of such an $h$ follows from $v \in V(C_i)$). If $\sigma(f_i, v) = -\sigma(h, v)$, then take $e = h$ and we are done, so assume $\sigma(f_i, v) = \sigma(h, v)$. Let $\vh$ be the orientation of $h$ where 
    $v$ is its tail. Since $h$ is trail-undirectable, there is an $r$--$\vh$~trail $T$ in $B$. Let $\vg$ be the edge just before $h$ in $T$. Since $\sigma(f_i, v) = \sigma(h, v)$, the trail $S \circ \gv$ is an $r$--$\gv$~trail in $B$. Therefore, $g$ is trail-undirectable, incident with $f_i$ at $v$, and $\sigma(f_i, v) = -\sigma(g, v)$, and we take $e = g$.

    Let $\ev$ be the orientation of $e$ with head $v$. Since $e$ is trail-undirectable, there is an $r$--$\ev$~trail $T$ in $B$.  If $\vf_i \not \in T$ then $T \circ \fv_i$ is an $r$--$\fv_i$~trail in $B$, contradicting that $\vf_i$ is trail-directable, so $\vf_i \in T$. Let $\hat{T}$ be the terminal segment of $T$ starting after $\vf_i$. Now $\vf_i \circ \hat{T}$ and $\vf_i \circ \hat{T}^{-1}$ are both trails, witnessing that $E(\hat{T}) \subseteq F$. This proves \cref{F-nonempty}.
    \end{claimproof}

    \begin{claim}\label{v-reachable}
        For every $w\in V(F)$  there exist an $\vf_i$--$(w,-)$~trail $T^{(w,-)}$ and an $\vf_i$--$(w,+)$~trail $T^{(w,+)}$ with all edges in $F \cup \{f_i\}$.
    \end{claim}
    \begin{claimproof}
        Pick an $h\in F$ that is incident with  $w$, let $\vh$ be the orientation of $h$ where $w$ is its tail and let $\alpha:= \sigma(h,w)$. Since $h \in F$, there exist an $\vf_i$--$\vh$~trail $T$ and an $\vf_i$--$\hv$~trail $T'$ whose edges are contained in $F \cup \{f_i\}$. Then the trail $T^{(w, \alpha)} := T'$ ends in $(w, \alpha)$ and has all edges in $F \cup \{f_i\}$. Furthermore, the trail $T^{(w, -\alpha)}$ obtained from $T$ by removing $h$ ends in $(w, - \alpha)$ and has all edges in $F \cup \{f_i\}$.
    \end{claimproof}

\begin{claim}\label{V(F) and S}
   $S$ meets $V(F)$ only at its last vertex $v$.  
\end{claim}
\begin{claimproof}
Note that $S$ meets $v$ only once by its definition. Let $w$ be the first vertex on $S$ that is in $V(F)$  and let $S'$ be the initial segment of $S$ until the first appearance of $w$ in $S$. Suppose for a contradiction that $w\neq v$. Then $\vf_i \notin S'$ and thus for a suitable sign $\alpha$, 
$S'\circ (T^{(w, \alpha)})^{-1}$ is an $r$--$\fv_i$~trail which contradicts the directabality of $f_i$. 
\end{claimproof}
\begin{claim}\label{VF=Ci}
   $V(F)=V(C_i)$.  
\end{claim}
\begin{claimproof}
It follows from \cref{V(F) and S}  that  whenever $T$ is a trail starting with $\vf_i$ and with all edges in $F \cup \{f_i\}$, then $(S - \vf_i) \circ T$ is also a trail.  By combining this with the definition of $F$ we conclude that $F$ consists of trail-undirectable edges belonging to the same trail-undirectable component. Since $v \in V(F)$, this means $F \subseteq E(C_i)$ and hence $V(F)\subseteq V(C_i)$.    Suppose for a contradiction that $V(C_i) \setminus V(F)\neq \emptyset$.  Since $C_i$ is a component of trail-undirectable edges, and $v\in  V(F)$, we can choose an  $x\in V(C_i) \setminus V(F)$ that is incident with a trail-undirectable edge $e$ whose other endpoint $y$ is in $V(F)$. Let $\ve$ be the orientation of $e$ with tail $x$ and head $y$. Since $e$ is trail-undirectable, $B$ contains an $r$--$\ve$~trail $R$. 

We claim that $\vf_i \in R$. Suppose $\vf_i \notin R$ and let $R'$ be the initial segment of $R$ until its first vertex $z$ in $V(F)$. Then for a suitable sign $\alpha$, $R' \circ (T^{(z, \alpha)})^{-1}$ is an $r$--$\fv_i$~trail, contradicting the fact that $f_i$ is trail-undirectable. This contradiction shows that $\vf_i \in R$. 

Let $w$ be the last vertex of $R-\ve$ that is in $V(F)$. Since $R-\ve $ goes through $v \in V(F)$, the vertex $w$ is well defined. Let $T$ be the terminal segment of $R$ starting at the last appearance of $w$ in it. Then $e$ witnesses $F \subsetneq F \cup E(T)$. Therefore, in order to get a contradiction to the maximality of $F$, it is enough to show that for every $g \in E(T)$, there is an $\vf_i$--$\vg$~trail and an $\vf_i$--$\gv$~trail with all edges in $F \cup E(T) $. For a suitable sign $\alpha$, $T^{(w,\alpha)}\circ T$ is a trail (see \cref{v-reachable}). For suitable sign $\beta$, $T^{(y,\beta)}\circ T^{-1}$ is a trail. These two trails together witness the desired property.
    %
\end{claimproof}

\begin{claim}
    $F$ is the set of edges with both endpoints in $C_i$. 
\end{claim}
\begin{claimproof}
It is clear from \cref{VF=Ci} that every edge in $F$ has both endpoints in $C_i$. Let $e$ be any edge with both endpoints in $C_i$. Let $\ve$ be an orientation of $e$ and let $x$ be the tail and $y$ the head of $\ve$. Let $x$ and $y$ be the endpoints of~$e$. Suppose for a contradiction that $e \not \in F $. Because of \cref{VF=Ci} we can apply \cref{v-reachable} with $x$ as well as with $y$ and construct the trails $T^{(x, -\sigma(e,x))} \circ \ve $ and $T^{(y, -\sigma(e,y))} \circ \ev$. But then $F \cup \{ e\}$ satisfies the conditions on $F$, contradicting the maximality of $F$. 
\end{claimproof}

  The proof of \ref{itm:3.1ii} and \ref{itm:3.1iii} is complete. It remains to show the uniqueness part of \ref{itm:3.1i}, i.e.\ that $f_i$ is the unique trail-directable edge such that $\vf_i$ has head in $C_i$. Suppose for contradiction that there exists a trail-directable edge $f_i' \neq f_i$ such that $\vec{f}_i'$  has head $x$ in $C_i$. Then $(S - \vec{f}_i) \circ T^{(x, -\sigma(f_i',x))} \cev{f}_i'$ is a $r$--$\cev{f}_i'$~trail, contradicting that $f'_i$ is trail-directable.
\end{proof}

\begin{corollary}
\label{lem:component_properties_root}
    Let $B$ be a bidirected graph rooted at $r$ and let $C$ be a nontrivial trail-undirectable component containing $r$. Then the following holds: 
    \begin{enumerate}[label=(\alph*)]
        \item\label{itm:3.2i} there is no trail-directable edge $f$ such that its natural orientation $\vf$ has head in $C$;
        \item\label{itm:3.2ii} for every $e \in E(B[C])$, there is an $r$--$\ve$~trail and an $r$--$\ev$~trail with all internal vertices in $C$; and
        \item\label{itm:3.2iii} for every vertex $w \in C \setminus \{r\}$, there is an $r$--$(w,-)$~trail $R^{(w,-)}$ and an $r$--$(w,+)$~trail $R^{(w,+)}$ in $B$ with all edges in $E(C)$.
    \end{enumerate}
\end{corollary}
\begin{proof}
    We can assume without loss of generality that all edges incident with $r$ have sign $+$ at $r$ since switching signs at $r$ does not affect what counts as an $r$-trail.
    Let $\Tilde{B}$ be the bidirected graph obtained from $B$ by adding a new vertex $\hat{r} $ and an edge $\hat{r} r$ with sign $-$ at $r$ and let $\hat{r}$ be the root of $\Tilde{B}$.
    Every $r$-trail in $B$ turns into an $\hat{r}$-trail in $\Tilde{B}$ by adding the edge $\hat{r} r$.
    Conversely, every $\hat{r}$-trail in $\Tilde{B}$ turns into an $r$-trail in $B$ by removing the edge $\hat{r} r$.
    Thus the trail-undirectable edges of $\Tilde{B}$ are precisely the trail-undirectable edges of $B$ and the orientations of the trail-directable edges coincide in $B$ and $\Tilde{B}$.
    In particular, $C$ is a trail-undirectable component of $\Tilde{B}$.

    Since $\hat{r} \notin V(C)$, we can apply~\cref{lem:component_properties} to $C$ in $\Tilde{B}$.
    Then $\hat{r} r$ is the unique trail-directable edge of $\Tilde{B}$ with head in $V(C)$, which implies that there is no trail-directable edge of $B$ with head in $V(C)$.
    Furthermore, by removing the edge $\hat{r} r$ from the trails provided by~\cref{lem:component_properties} we obtain the desired trails in $B$.
    \end{proof}

For the end of this subsection let us further assume that $B$ is trail-reachable. We note that this assumption does not restrict results about connectivity from $r$: Every rooted bidirected graph becomes trail-reachable by deleting all non-trail-reachable edges. Furthermore, deleting all non-trail-reachable edges does not affect what counts as an $r$-trail, moreover, trail-reachable edges remain trail-reachable after the deletions.

Let $C_i$ and $f_i$ be as in~\cref{lem:component_properties}. We write $\vf_i$ for the natural orientation of $f_i$ and let $c_i$ be the head of $\vf_i$. Furthermore, for any component $C_i$ containing $r$ we set $c_i:= r$ if $V(C_i)$ contains $r$. A vertex $x \in V(B)$ is \emph{trail-solid} if $x \notin \bigcup_{i \in [k]} V(C_i) - c_i$.

 Let $\hat{B}$ be the bidirected graph formed by contracting the component $C_i$ onto the vertex~$c_i$ for $i = 1, \hdots, k$.  
Note that $V(\hat{B})$ is exactly the set of trail-solid vertices of $B$ and $E(\hat{B})$  corresponds to the set of trail-directable edges of $B$. (However, we stress that $\hat{B}$ is not a substructure of $B$ in the intuitive sense. Specifically, two edges $e$ and $f$ that are incident in $\hat{B}$ at vertex $c_i$ for $i = 1, \hdots, k$ need not be incident in $B$.)
Now let $\trailskeleton{B}$ be the rooted directed graph formed from $\hat{B}$ by giving every edge its natural orientation. Observe that by \cref{lem:component_properties}, the vertex $c_i$ has in-degree one for all $i = 1, \hdots, k$ and the unique in-edge incident with $c_i$ is $f_i$. We call $\trailskeleton{B}$ the {\em trail-skeleton} of $B$. See~\cref{fig:edge-decomp} for an example.

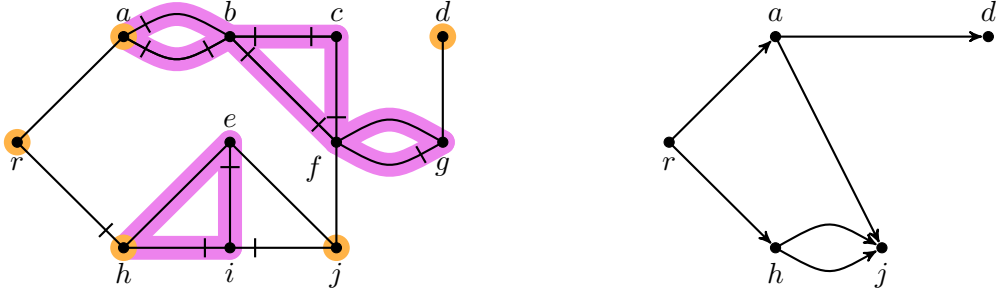
\begin{figure}
\centering
\begin{subfigure}[b]{0.54\textwidth}
\centering
\begin{tikzpicture}
\coordinate (r) at (0,1.4);
\coordinate (a) at (1.4,2.8);
\coordinate (b) at (2.8,2.8);
\coordinate (c) at (4.2,2.8);
\coordinate (d) at (5.6,2.8);
\coordinate (e) at (2.8,1.4);
\coordinate (f) at (4.2,1.4);
\coordinate (g) at (5.6,1.4);
\coordinate (h) at (1.4,0);
\coordinate (i) at (2.8,0);
\coordinate (j) at (4.2,0);

\foreach \i/\j in {h/e,e/i,h/i,b/c,c/f,b/f}{
    \draw[line width=9pt, LavenderMagenta, line cap=round] (\i) to (\j);
}

\foreach \i/\j in {a/b,f/g}{
    \draw [line width=9pt, LavenderMagenta, line cap=round] (\i) ..controls +($-0.5*(\i)+0.5*(\j)+(0,+0.4)$) .. (\j);
    \draw [line width=9pt, LavenderMagenta, line cap=round] (\i) ..controls +($-0.5*(\i)+0.5*(\j)+(0,-0.4)$) .. (\j);
}
\foreach \i in {r,a,d,h,j}{
    \fill[PastelOrange] (\i) circle (5pt);
}

\foreach \i in {a,b,c,d,e}{
    \node[dot] at (\i) [label=above:$\i$] {};
}
\foreach \i in {f}{
    \node[dot] at (\i) [label=below left:$\i$] {};
}
\foreach \i in {h,i,j,g,r}{
    \node[dot] at (\i) [label=below:$\i$] {};
}
\foreach \i/\j in {r/a, h/e, e/j, f/j, d/g}{
    \draw[thick] (\i) to (\j);
}
\foreach \i/\j in {r/h, h/i, i/e, j/i, c/f, b/c, c/b, b/f, f/b}{
    \draw[-e+8,thick] (\i) to (\j);
}
\foreach \i/\j in {a/b, b/a, f/g}{
     \draw [-e+8,thick] (\i) ..controls +($-0.5*(\i)+0.5*(\j)+(0,-0.4)$) .. (\j);
}
\foreach \i/\j in {b/a}{
    \draw [-e+8,thick] (\i) ..controls +($-0.5*(\i)+0.5*(\j)+(0,+0.4)$) .. (\j);
}
\foreach \i/\j in {f/g}{
    \draw [thick] (\i) ..controls +($-0.5*(\i)+0.5*(\j)+(0,+0.4)$) .. (\j);
}

\end{tikzpicture}
\caption{A bidirected graph $B$ rooted in $r$. The trail-solid vertices of $B$ are marked orange and the trail-undirectable edges of $B$ are marked purple.}
\end{subfigure}
\hfill
\begin{subfigure}[b]{0.45\textwidth}
\centering
\begin{tikzpicture}
\coordinate (r) at (0,1.4);
\coordinate (a) at (1.4,2.8);
\coordinate (d) at (4.2,2.8);
\coordinate (h) at (1.4,0);
\coordinate (j) at (2.8,0);

\foreach \i in {a,d}{
    \node[dot] at (\i) [label=above:$\i$] {};
}
\foreach \i in {r,h,j}{
    \node[dot] at (\i) [label=below:$\i$] {};
}
\foreach \i/\j in {r/a,r/h,a/d,a/j}{
    \draw[arc,thick] (\i) to (\j);
}
\draw [arc,thick] (h) ..controls +($-0.5*(h)+0.5*(j)+(0,+0.4)$) .. (j);
\draw [arc,thick] (h) ..controls +($-0.5*(h)+0.5*(j)+(0,-0.4)$) .. (j);
\end{tikzpicture}
\caption{The trail-skeleton of $B$. \newline \newline}
\end{subfigure}
    \caption{A rooted bidirected graph $B$ and its trail-skeleton.}
    \label{fig:edge-decomp}
\end{figure}

\begin{lemma}\label{lem:vertexincomponent}
    A vertex $v$ is contained in a trail-undirectable component if and only if there is an $r$-$(v,\alpha)$~trail for every $\alpha \in \{+,-\}$.
\end{lemma}
\begin{proof}
    The vertex $v$ is contained in a trail-undirectable component if and only if $v$ is incident with an undirectable edge $e$. If $v$ is incident with an undirectable edge $e$, there is an $r$--$\ve$~trail and an $r$--$\ev$~trail, which contain the desired trails.
    
    Now, we assume that there is an $r$--$(v,\alpha)$~trail for every $\alpha \in \{+,-\}$. Let $T$ be the shortest $r$--$(v,+)$~trail. If $T$ contains an $r$--$(v,-)$~trail $T'$, let $\vf$ be the terminal edge of $T$. Then the trails $T$ and $T' \circ \fv$ witness that $f$ is undirectable. Otherwise, $v$ is not an internal vertex of $T$. We pick an $r$--$(v,-)$~trail $T''$ and let $\vf$ be the terminal edge of $T''$. Then the trails $T \circ \vf$ and $T''$ witness that $f$ is undirectable.
\end{proof}

\begin{lemma}\label{cor:trail-solid}
    A vertex $v$ is trail-solid if and only if one of the following holds
    \begin{itemize}
        \item $v = r$;
        \item there is a trail-directable edge $e$ such that $\ve$ has head $v$; or 
        \item $v$ is not incident with an undirectable edge.
    \end{itemize}
\end{lemma}
\begin{proof}
    A vertex $v \in V(B) \setminus \bigcup_{i \in [k]} V(C_i)$ is trail-solid and not incident with an undirectable edge.
    A vertex $w \in \bigcup_{i \in [k]} V(C_i)$ is trail-solid if and only if $w = c_i$ for some $i \in [k]$. Furthermore for $i \in [k]$ we have $w = c_i$ if and only if $w = r$ or $\vf_i$ has head in $c_i$, by~\cref{lem:component_properties,lem:component_properties_root}.
\end{proof}

\subsection{Transferring trails between the bidirected graph and its trail-skeleton}
Next we explain the correspondence between trails in $B$ and in its trail-skeleton. Given an $r$-trail $T$ in $B$, we define $\mathbb{o}(T)$ to be the unique $r$-trail in $\trailskeleton{B}$ which follows the trail-directable edges of $T$ in their order of appearance along $T$. Of course, we need to show that there really is such a trail.

\begin{lemma}\label{lem:bijection_trail_to_trail_skeleton1}
    Let $B$ be a trail-reachable bidirected graph rooted at $r$. Let $T$ be an $r$-trail of $B$. Then there is an $r$-trail $\mathbb{o}(T)$ in $\trailskeleton{B}$ as described above.
\end{lemma}
\begin{proof}
    By \Cref{lem:vertexincomponent,cor:trail-solid}, the trail $T$ can be partitioned into trail segments $(T_1, T_2, \ldots, T_\ell)$ such that for every odd $i$, $T_i$ is a sequence of trail-directable edges in $T$, and for every even $i$, $T_i$ is contained in a unique trail-undirectable component. From the definition of trail-skeleton, for odd $i$, $T$ witnesses that $T_i$ corresponds to a directed path $D[T_i]$ in $\trailskeleton{B}$ oriented as in $T$, and the unique trail-undirectable component containing $T_{i+1}$ witnesses that the last vertex of $D[T_i]$ is equal to the first vertex of $D[T_{i+2}]$.
\end{proof}

\begin{lemma}\label{lem:bijection_trail_to_trail_skeleton2}
    Let $B$ be a trail-reachable bidirected graph rooted at $r$. For every $r$-trail $S$ of $\trailskeleton{B}$, there exists an $r$-trail $T$ of $B$ with $\mathbb{o}(T)=S$.
\end{lemma}
\begin{proof}
   From the definition of trail-skeleton, $V(\trailskeleton{B})$ is exactly the set of trail-solid vertices of $B$. First we will construct, for each edge $\vec{uv}$ in $\trailskeleton{B}$, a $u$--$v$~trail in $B$. If $u \in V(B) \setminus \bigcup_{i \in [k]} V(C_i)$, then $\vec{uv}$ also exists in $B$ by \cref{cor:trail-solid}. Otherwise, $u$ is in some trail-undirectable component $C_i$. Then by \cref{cor:trail-solid} and \cref{lem:component_properties}--\cref{itm:3.1i}, $u$ is either $r$ or the head of $\vf_i$. If $\vec{uv}\notin B$, then there exists a vertex $w$ in $C_i$ other than $u$ satisfying $\vec{wv}\in B$. By \cref{lem:component_properties}-\cref{itm:3.1iii} and \cref{lem:component_properties_root}-\cref{itm:3.2iii}, there exists a $u$--$\vec{wv}$~trail in $B$ with all internal vertices in $C_i$.
   Since all trail-undirectable components are pairwise disjoint, for any $r$-trail in $\trailskeleton{B}$, by concatenating all these trails we can find an $r$-trail in $B$ that satisfies the required condition. 
\end{proof}

\begin{lemma}\label{lem:paths_in_path_reachable}
    Let $B$ be a path-reachable bidirected graph rooted at $r$. For every $r$-path $P$ of $\trailskeleton{B}$, there exists an $r$-path $Q$ of $B$ with $\mathbb{o}(Q)=P$.
\end{lemma}
\begin{proof}
    Similar to \cref{lem:bijection_trail_to_trail_skeleton2}, we will construct, for each edge $\vec{uv}$ in $\trailskeleton{B}$, a $u$--$v$~path $P_0$ in $B$. If $\vec{uv}$ is in $B$, we just take it. Otherwise, it follows that $u$ is in some trail-undirectable component $C_i$ of $B$ and there exists a $u$--$v$~trail $T$ in $B$ with $\vec{wv}\in T$ and $w\in V(C_i)$. Let $g$ be the trail-directable edge of $T$ from $w$ to $v$. Since every edge in $B$ is path-reachable, there is a path $P_{g}$ from $r$ ending in $g$. 
    If $u=r$, then we can set $P_0:=P_{g}$. Otherwise, $u$ is the head of $\vf_i$. Since $\vf_i$ is the unique edge of $B$ with head $u$ in $V(C_i)$ by \cref{lem:component_properties}--\cref{itm:3.1i}, it follows that $\vf_i \in P_{g}$, and so we can set $P_0:=uP_{g}$. Since the internal vertices of $P_0$ are contained in $V(C_i)$ and all trail-undirectable components are pairwise disjoint, by concatenating all these paths we can find an $r$-path $Q$ in $B$ with $\mathbb{o}(Q)=P$. This completes the proof.
\end{proof}

\subsection{Edge-disjointness}
The definition of $\mathbb{o}(T)$ implies:
\begin{remark}\label{rem:trail-skeleton-edge-disjoint}
    Let $B$ be a trail-reachable bidirected graph and let $T_1$ and $T_2$ be two edge-disjoint $r$~trails of $B$. Then $\mathbb{o}(T_1)$ and $\mathbb{o}(T_2)$ are two edge-disjoint $r$-trails of $\trailskeleton{B}$.
\end{remark}
\noindent
In this subsection we show that the converse holds true if $B$ is edge-clean:
A bidirected graph $B$ rooted at $r$ is called {\em edge-clean} if $r$ is not contained in a trail-undirectable component of $B$. In other words, all edges incident with $r$ are trail-directable.
\begin{lemma}\label{lem:edge-clean_characterisation}
    Let $B$ be a bidirected graph rooted at $r$. Then $B$ is edge-clean if and only if $B$ does not contain a nontrivial $r$--$r$~trail.  
\end{lemma}
\begin{proof}
 First, assume that $B$ is edge-clean and suppose for a contradiction that $T$ is a nontrivial $r$--$r$~trail in $B$. Then $T^{-1}$ is also a nontrivial $r$--$r$~trail, so every edge of $T$ is trail-undirectable and $r$ is contained in a trail-undirectable component, contradicting that $B$ is edge-clean. Next, assume that $B$ is not edge-clean, so $r$ is in a trail-undirectable component $C$. Let $e \in E(C)$ be a trail-undirectable edge incident with $r$ in $V(C)$. Then by definition there is an $r$-$r$-trail with final edge $e$.
\end{proof}

An $r$-trail $T$ of $B$ is {\em proper} if $T$ contains at least one trail-directable edge.
\begin{lemma}\label{lem:trails_sharing_edge}
    Let $B$ be a edge-clean bidirected graph and let $T$ and $T'$ be proper $r$-trails of $B$. If $T$ and $T'$ share an edge, then they share a trail-directable edge. 
\end{lemma}
\begin{proof}
    Let $e$ be an edge in $E(T) \cap E(T')$. We can assume that $e$ is trail-undirectable, since otherwise we are done.
    Let $C_i$ be the trail-undirectable component containing $e$.
    
    Since $B$ is edge-clean, $r \notin V(C_i)$ and we can apply~\cref{lem:component_properties} to $C_i$.
    Thus there exists a unique trail-directable edge $\vf_i$ with head in $V(C_i)$.
    Note that the first edge of $T$ with an endpoint in $V(C_i)$ is trail-directable with head in $V(C_i)$ and thus is $f_i$.
    Similarly, $f_i$ is the first edge of $T'$ with an endpoint in $V(C_i)$. Thus $f_i$ is a trail-directable edge in $E(T) \cap E(T')$.
\end{proof}

\begin{corollary}\label{cor:trail-skeleton-edge-disjoint}
    Let $B$ be an edge-clean, trail-reachable bidirected graph rooted at $r$ and let $T_1$ and $T_2$ be two edge-disjoint $r$~trails of $B$. If $\mathbb{o}(T_1)$ and $\mathbb{o}(T_2)$ are edge-disjoint then so are $T_1$ and $T_2$.
\end{corollary}
\begin{proof}
    Suppose not for a contradiction. Since $T_1$ and $T_2$ are not edge-disjoint, then they share a trail-directable edge $\ve$ by~\cref{lem:trails_sharing_edge}, which implies that $\ve \in \mathbb{o}(T_1)\cap \mathbb{o}(T_2) = S_1\cap S_2$, which is the desired contradiction.
\end{proof}

\subsection{Connectivity}
We compare the rooted connectivities of $B$ and $\trailskeleton{B}$ in the following lemmata.

Recall that, given vertices $v$ and $w$ of a directed graph $D$, we write $\lambda_D(v,w)$ for the maximum number of edge-disjoint directed $v$--$w$~paths.
In the same way, given vertices $v$ and $w$ of a bidirected graph $B$, let $\lambda_B^{\textrm{path}}(v,w)$ be the maximum number of edge-disjoint $v$--$w$~paths and let $\lambda_B^{\alpha, \textrm{ path}}(v,w)$ be the maximum number of edge-disjoint $v$--$(w,\alpha)$~paths for $\alpha \in \{+,-\}$. Similarly, we define $\lambda_B^{\textrm{trail}}(v,w)$ and $\lambda_B^{\alpha, \textrm{ trail}}(v,w)$.
\begin{lemma}\label{lem:disjoint_paths_skeleton}
    Let $B$ be an edge-clean, trail-reachable bidirected graph rooted at $r$ and let $x \in V(\trailskeleton{B})$. Then
    $\lambda_B^{\textrm{trail}}(r, x) = \lambda_{\trailskeleton{B}}(r, x)$. Furthermore, if $B$ is path-reachable, then $\lambda_B^{\textrm{path}}(r, x) = \lambda_{\trailskeleton{B}}(r, x)$.
\end{lemma}
\begin{proof}
By~\cref{rem:trail-skeleton-edge-disjoint,cor:trail-skeleton-edge-disjoint}, it holds that $\lambda_B^{\textrm{trail}}(r, x) = \lambda_{\trailskeleton{B}}(r, x)$, since every $r$--$x$~trail in $\trailskeleton{B}$ includes an $r$--$x$~path. If $B$ is path-reachable then by combining this with \cref{lem:paths_in_path_reachable}, we can deduce $\lambda_B^{\textrm{path}}(r, x) = \lambda_{\trailskeleton{B}}(r, x)$.
\end{proof}

\begin{lemma}\label{lem:trail-solid-one-sign}
    Let $B$ be a trail-reachable bidirected graph rooted at $r$ and let $x \in V(\trailskeleton{B}) \setminus \{r\}$. Then there is $\alpha \in \{+,-\}$ such that there is no $r$--$(x,\alpha)$~path in $B$, i.e.\ $\lambda_B^{\alpha, \textrm{ path}}(r, x) = 0$.
\end{lemma}
\begin{proof}
    Suppose not, and let $P$ be an $r$--$(x,+)$~path and $Q$ an $r$--$(x,-)$~path in $B$. Furthermore, let $e$ be the terminal edge of $P$. Then $P$ and $Q + e$ witness that $e$ is trail-undirectable. Thus $x$ is contained in an undirectable component $C_i$ of $B$. Since $x \in V(\trailskeleton{B}) \setminus \{r\}$ we have $x = c_i$, where $c_i \in V(C_i)$ for some trail-undirectable component $C_i$ that does not contain $r$. Thus $\lambda_B^{\alpha, \textrm{ path}}(r,x) =0$ for some $\alpha \in \{+,-\}$ by~\cref{lem:component_properties}, a contradiction.
\end{proof}

\begin{lemma}\label{lem:disjoint_paths_skeleton_inside}
    Let $B$ be an edge-clean, trail-reachable bidirected graph rooted at $r$. Let $C_i$ be a trail-undirectable component of $B$. Then $\lambda_{\trailskeleton{B}}(r, c_i) = 1$ and $\lambda_B^{\text{trail}}(r, x) = 1$ for every $x \in V(C_i)$. Furthermore, if $B$ is path-reachable, then $\lambda_B^{\text{path}}(r, x) = 1$ for every $x \in V(C_i)$. 
\end{lemma}
\begin{proof}
    By~\cref{lem:component_properties}, there exists a unique edge $f_i$ such that $\vf_i$ has head in $V(C_i)$. 
    Thus every $r$--$x$~trail for $x \in V(C_i)$ contains $f_i$, which implies that $\lambda_B^{\text{trail}}(r, x) = 1$ for every $x \in V(C_i)$.
    We deduce that $\lambda_{\trailskeleton{B}}(r, c_i) = \lambda_{B}^{\textrm{trail}}(r, c_i) = 1$ by applying the first assertion of~\cref{lem:disjoint_paths_skeleton} with $x = c_i$.
    Furthermore, if $B$ is path-reachable, then $1 \leq \lambda_{{B}}^{\textrm{path}}(r, x) \leq \lambda_B^{\textrm{trail}}(r, x) = 1$ for every $x \in V(C_i)$.
\end{proof}

\section{Application: Edge version of Menger's Theorem} \label{sec:edge-menger}
Bowler et al.~\cite{bowler2023mengertheorembidirected}*{Theorem 5.1} showed that Menger's theorem for edge-disjoint paths holds true in edge-clean bidirected graphs.
%
In this section, we generalize the following strong version of Menger's theorem to edge-clean bidirected graphs:
\begin{theorem}\cite{frank2011connections}*{Theorem 2.5.1}\label{thm:directed_strong_menger}
    Let $D$ be a directed graph and let $r, x \in V(D)$. Then 
    \[ \lambda_D(r, x) = \min \left\lbrace  \left| \delta_D(X)
            \right|\ :\ 
             x\in X 
             \subseteq V(D)\setminus \{ r \}\right\rbrace. \]
\end{theorem}
\noindent

Given a bidirected graph $B$ rooted at $r$ and some set $X \subseteq V(B) \setminus \{r\}$, we define $\delta_{r,B}^{\mathsf{trail}} (X )$ and $\delta_{r,B}^{\mathsf{path}} (X )$ to be the sets of edges $e \in E(B)$ for which there exists an orientation $\ve$ such that the head of $\ve$ is in $X$, the tail of $\ve$ is not in $X$ and there exists an $r$--$\ve$~trail or an $r$--$\ve$~path, respectively.
Furthermore, given a directed graph $D$, $\delta_D(X)$ is the set of directed edges in $D$ whose tail is in $V(D) \setminus X$ and whose head is in $X$.

We begin by proving a version for trails:
\begin{theorem}\label{thm:strong_menger}
    Let $B$ be an edge-clean bidirected graph rooted at $r$ and let $x \in V(B) \setminus \{r\}$. Then 
    \[ \lambda_B^{\textrm{trail}}(r, x) = \min \left\lbrace  \left| \delta_{r,B}^{\mathsf{trail}} (X )
            \right|\ :\ 
             x\in X 
             \subseteq V(B)\setminus \{ r \}\right\rbrace. \]
\end{theorem}
\noindent
We note that~\cref{thm:strong_menger} becomes false in general if we omit the condition of edge-cleanness~\cite{bowler2023mengertheorembidirected}*{Figure 2}.
Our proof of~\cref{thm:strong_menger} applies~\cref{thm:directed_strong_menger} to the edge-decomposition.

\begin{proof}
To show the inequality ``$\leq$'', let $\mathcal{P}$ be a set of edge-disjoint $r$--$x$~trails of size $\lambda_{B}^{\textrm{trail}}(r, x)$ in $B$. For every set $X \subseteq V(B) \setminus \{r\}$, the set $\delta_{r, B}^\mathsf{trail}(X)$ contains at least one edge of every path in $\mathcal{P}$.  Since the paths in $\mathcal{P}$ are edge-disjoint, it follows that \[\lambda_B^{\textrm{trail}}(r, x) \leq \min \left\lbrace  \left| \delta_{r,B}^{\mathsf{trail}} (X )
            \right|\ :\ 
             x\in X 
             \subseteq V(B)\setminus \{ r \}\right\rbrace .\]
             
To show the inequality ``$\geq$'', let $D$ be the trail-skeleton of $B$.
First, assume that $x \in V(D)$. By~\Cref{lem:disjoint_paths_skeleton}, $ \lambda_B^{\textrm{trail}}(r, x) = \lambda_D(r, x)$.        
Furthermore, by~\cref{thm:directed_strong_menger}, \[ \lambda_D(r, x) = \min \left\lbrace  \left| \delta_{D} (X )
            \right|\ :\ 
             x\in X 
             \subseteq V(D)\setminus \{ r \}\right\rbrace.\] Let $X \subseteq V(D) \setminus \{r\}$ be such that $x \in X$ and $\lambda_D(r, x) = |\delta_{D}(X)|$. By~\Cref{lem:bijection_trail_to_trail_skeleton2}, there exists a set $X_B:= X \cup \bigcup_{C_i: V(C_i) \cap X \neq \emptyset} V(C_i)$ in $V(B)$ such that $X_B \cap V(D) = X$ and $|\delta_{D}(X)| = |\delta_{r, B}^{\mathsf{trail}}(X_B)|$ since every edge of $B$ with precisely one endpoint in $X_B$ is trail-directable. We can conclude $ \lambda_B^{\textrm{trail}}(r, x) \geq \min \left\lbrace  \left| \delta_{r,B}^{\mathsf{trail}} (X )
            \right|\ :\ 
             x\in X 
             \subseteq V(B)\setminus \{ r \}\right\rbrace$.
    
             Next, assume that $x \in V(B) \setminus V(D)$.
             By~\cref{lem:disjoint_paths_skeleton_inside}, $\lambda_B^{\textrm{trail}}(r,x)=1$.
             Furthermore, there exists an undirectable component $C_i$ with $x \in V(C_i)$ and, by~\cref{lem:component_properties}, $|\delta_{r, B}^{\mathsf{trail}}(C_i)| = 1$.
             This implies
             $\lambda_B^{\textrm{trail}}(r, x) \geq \min \left\lbrace  \left| \delta_{r,B}^{\mathsf{trail}} (X )
            \right|\ :\ 
             x\in X 
             \subseteq V(B)\setminus \{ r \}\right\rbrace$, 
             which completes the proof.
    \end{proof}

    Now we turn our attention to paths:
    \begin{theorem}\label{thm:strong_menger_path}
     Let $B$ be an edge-clean bidirected graph rooted at $r$ and let $x \in V(B)$. Then  
        \[ \lambda_B^{\textrm{path}}(r, x) = \min \left\lbrace  \left| \delta_{r,B}^{\mathsf{path}} (X )
            \right|\ :\ 
             x\in X 
             \subseteq V(B)\setminus \{ r \}\right\rbrace. \]  
\end{theorem}
\noindent
Also~\cref{thm:strong_menger_path} is false in general if we omit the condition of edge-cleanness~\cite{bowler2023mengertheorembidirected}*{Figure 2}.
    \begin{proof}
    We assume without loss of generality that for every $v \in V(B)$ there exists an $r$--$v$~path in $B$.
    Let $B'$ be the subgraph of $B$ induced by the path-reachable edges of $B$.
    Note that $V(B') = V(B)$ and that $B'$ is path-reachable.

    Given a path $P$ in $B$, the path $P$ witnesses that every edge of $P$ is path-reachable, so $P$ is a path of $B'$. Furthermore, since $B'$ is a subgraph of $B$, every path in $B'$ is a path in $B$. Therefore, $ \lambda_{B}^{\textrm{path}}(r, x) = \lambda_{B'}^{\textrm{path}}(r, x)$ and $\delta_{r, B}^{\mathsf{path}}(X) = \delta_{r, B'}^{\mathsf{path}}(X)$ for every $X \subseteq V(B) \setminus \{r\}$ with $x \in X$.
    Thus it suffices to show that
    \[\lambda_{ B'}^{\textrm{path}}(r, x) = \min \left\lbrace  \left| \delta_{r,B'}^{\mathsf{path}} (X )
            \right|\ :\ 
             x\in X 
             \subseteq V(B') \setminus \{ r\}\right\rbrace. \]

    For every set $\mathcal{P}$ of edge-disjoint $r$--$x$~paths in $B'$ and for every set $X \subseteq V(B') \setminus \{r\}$ with $x \in X$, the set $\delta_{r, B'}^{\mathsf{path}}(X)$ contains at least one edge from every path in $\mathcal{P}$, so
        \[ \lambda_{ B'}^{\textrm{path}}(r, x) \leq \min \left\lbrace  \left| \delta_{r,B'}^{\mathsf{path}} (X )
            \right|\ :\ 
             x\in X 
             \subseteq V(B') \setminus \{ r\}\right\rbrace.\]

    By applying~\Cref{lem:disjoint_paths_skeleton,lem:disjoint_paths_skeleton_inside} to the path-reachable $B'$, we obtain $\lambda_{B'}^{\textrm{path}}(r, x) = \lambda_{B'}^{\textrm{trail}}(r, x)$.
    Furthermore, by~\cref{thm:strong_menger}, $$\lambda_{B'}^{\textrm{trail}}(r, x) = \min \left\lbrace  \left| \delta_{r,B'}^{\mathsf{trail}} (X )
            \right|\ :\ 
             x\in X 
             \subseteq V(B')\setminus \{ r \}\right\rbrace.$$
    \noindent
    Next, let $X \subseteq V(B) \setminus \{r\}$.
    Since every path of $B'$ is also trail of $B'$, it follows that $\delta_{r, B'}^{\mathsf{path}}(X) \subseteq \delta_{r, B'}^{\mathsf{trail}}(X)$, so
    \begin{equation*}
            \min \left\lbrace  \left| \delta_{r,B'}^{\mathsf{trail}} (X )
            \right|\ :\ 
             x\in X 
             \subseteq V(B')\setminus \{ r \}\right\rbrace \geq \min \left\lbrace  \left| \delta_{r,B'}^{\mathsf{path}} (X )
            \right|\ :\ 
             x\in X 
             \subseteq V(B')\setminus \{ r \}\right\rbrace,
    \end{equation*} 
    which completes the proof.    
\end{proof}

\begin{corollary} \label{thm:menger_edge}
    Let $B$ be an edge-clean bidirected graph rooted at $r$ and let $x \in V(B)$. Then the maximum number of edge-disjoint $ r$--$ x $~paths ($r$--$x$~trails) in $ B $ is equal to the minimum size of an edge set $F \subseteq E(B)$ such that there is no $r$--$x$~path (no $r$--$x$~trail) in $B - F$. 
\end{corollary}
\noindent
We remark that this section provides a simple proof for the variant of Menger's theorem that was discovered in~\cite{bowler2023mengertheorembidirected}*{Theorem 5.1}.

\section{Application: Flames} \label{sec:edge-flame}

An  $r$-rooted directed graph $D=(V,E)$ is a \emph{flame} if for every $v\in V\setminus \{r\}$ the local edge-connectivity $\lambda_D(r,v)$, i.e.\ the maximum number of edge-disjoint $r$--$v$~paths, is equal to the in-degree of $v$. Lovász proved that every rooted directed graph admits a flame as a spanning subgraph witnessing all local edge-connectivities from the root:
    \begin{theorem}[Lovász, \cite{lovasz1973connectivity}*{Theorem 2}]\label{t: LovFlame}
    Every $r$-rooted directed graph $D$ admits an $r$-rooted flame $ F $ with $\lambda_F(r,v)=\lambda_D(r,v)$ for every $ v\in V(D) \setminus \{ r \} $.
   \end{theorem}

   For a generalization of~\cref{t: LovFlame} to edge-disjoint paths in bidirected graphs, the definition of flames cannot be transferred literally:
    On the one hand, the graph in \cref{fig:flamesB} shows that it is not sensible to consider edge-connectivity but rather signed edge-connectivity, i.e.\ the maximum number $\lambda_B^{\alpha,\textrm{path}}(r,v)$ of $r$-$(v,\alpha)$~paths for $\alpha \in \{+,-\}$.
    On the other hand, the bidirected graph in \cref{fig:flamesA} has a vertex $x$ of low signed edge-connectivity but many `ingoing' edges, i.e.\ edges in which an $r$--$x$~path ends, that are necessary to witness all local signed edge-connectivities. Thus for a statement like \cref{t: LovFlame} flames should not be defined by a local but rather by a global bound on the number of edges. 


    A directed graph $D$ is a flame if and only if the number of edges of $D$ whose head is not $r$ is bounded by $\sum_{v \in V(D) \setminus \{r\}} \lambda_D(r,v)$. In this way we generalize flames to bidirected graphs: 
    An $r$-rooted bidirected graph $B$ is a \emph{flame} if the number of edges of $B$ is bounded by $\sum_{v \in V(B) \setminus \{r\}} \sum_{\alpha \in \{+,-\}} \lambda_B^{\alpha \textrm{, path}}(r,v)$.
    We show that every edge-clean bidirected graph has a flame that witnesses all signed connectivities:


\begin{figure}
\centering
\begin{subfigure}[b]{0.49\textwidth}
\centering
\begin{tikzpicture}
\coordinate (r) at (0,0);
\coordinate (a) at (1.5,0);
\coordinate (b) at (3,0);
\coordinate (c) at (4.5,0.73);
\coordinate (d) at (4.5,-0.73);

\foreach \i in {r,b,c,d}{
    \node[dot] at (\i) [label=above:$\i$] {};
}
\foreach \i in {a,b,c,d}{
    \node[dot] at (\i) [] {};
}

\foreach \i/\j in {b/c}{
    \draw[thick] (\i) to (\j);
}
\foreach \i/\j in {r/a,d/b}{
    \draw[-e+8,thick] (\i) to (\j);
}
\draw [-e+8,thick] (a) ..controls +($-0.5*(a)+0.5*(b)+(0,-0.4)$) .. (b);
\draw [thick] (b) ..controls +($-0.5*(b)+0.5*(a)+(0,+0.4)$) .. (a);

\node[draw=none, fill=none] at (0,-1) {};
\node[draw=none, fill=none] at (0,1.7) {};

\end{tikzpicture}
    \caption{Both `ingoing' edges at $b$ are necessary to witness the connectivies at $c$ and $d$, but $\lambda_B^{\textrm{path}}(r,b)=1$.
    Furthermore, all five edges of the bidirected graph $B$ are necessary to witness the local edge-connectivities, but $\sum_{v \in V(B) \setminus \{r\}} \lambda_B^{\textrm{path}}(r,v) = 4$. 
    }
    \label{fig:flamesB}
\end{subfigure}
\hfill
\begin{subfigure}[b]{0.49\textwidth}
\centering
\begin{tikzpicture}
\coordinate (r) at (0,0);
\coordinate (a) at (1,0);
\coordinate (b) at (2.5,0);
\coordinate (c) at (3.5,0);
\coordinate (d) at (4.5,0);
\coordinate (x) at (2,1);

\foreach \i in {r,x}{
    \node[dot] at (\i) [label=above:$\i$] {};
}
\foreach \i in {a,b,c,d}{
    \node[dot] at (\i) [] {};
}

\foreach \i/\j in {x/b,x/c,x/d}{
    \draw[thick] (\i) to (\j);
}
\foreach \i/\j in {r/a, a/x}{
    \draw[-e+8,thick] (\i) to (\j);
}
\draw [-e+8,thick] (a) ..controls +($-0.5*(a)+0.5*(b)+(0,-0.4)$) .. (b);
\draw [-e+8,thick] (a) ..controls +($-0.5*(a)+0.5*(c)+(0,-0.8)$) .. (c);
\draw [-e+8,thick] (a) ..controls +($-0.5*(a)+0.5*(d)+(0,-1.2)$) .. (d);

\node[draw=none, fill=none] at (0,-1) {};
\node[draw=none, fill=none] at (0,1.7) {};

\end{tikzpicture}
    \caption{There are four `ingoing' edges at $x$ and $\lambda_B^{\textrm{path}}(r,x) \leq \lambda_B^{+, \textrm{ path}}(r,x) + \lambda_B^{-, \textrm{ path}}(r,x) = 2$. \newline \newline \newline}
    \label{fig:flamesA}
\end{subfigure}
    \caption{Two bidirected graphs in which every edge is necessary to witness all local signed edge-connectivities.}
\label{fig:flames}
\end{figure}

    \begin{restatable}{theorem}{EdgeFlameBidirected} \label{thm:edge_flame_bidirected}
    Every $r$-rooted edge-clean bidirected graph $B$ admits an $r$-rooted flame $ F $ with $\lambda_F^{\alpha, \textrm{ path}}(r,v)=\lambda_B^{\alpha, \textrm{ path}}(r,v)$ for every $ v\in V(B) \setminus \{ r \} $ and $\alpha \in \{+,-\}$.
   \end{restatable}
   \noindent
   We note that~\cref{thm:edge_flame_bidirected} is false in general for non-edge-clean bidirected graphs (see~\cref{fig:flame_edge-clean}).
    For the proof of~\cref{thm:edge_flame_bidirected} in~\cref{subsec:existence_edge_flames}, we will need not only our earlier decomposition but also an additional tool - the existence of a kind of `ear-decomposition' for bidirected graphs, which we now investigate in~\cref{subsec:new_ear_decomposition}.

    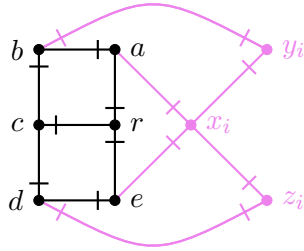
\begin{figure}
        \centering
        \begin{tikzpicture}
\coordinate (r) at (1,1);
\coordinate (a) at (1,2);
\coordinate (b) at (0,2);
\coordinate (c) at (0,1);
\coordinate (d) at (0,0);
\coordinate (e) at (1,0);
\coordinate (x_i) at (2,1);
\coordinate (y_i) at (3,2);
\coordinate (z_i) at (3,0);

\foreach \i/\j in {r/c,c/b,b/a,a/r,c/d,d/e,e/r}{
    \draw[-e+6,thick] (\i) to (\j);
}
\foreach \i/\j in {a/x_i,e/x_i,x_i/y_i,x_i/z_i}{
    \draw[-e+8,thick,LavenderMagenta] (\i) to (\j);
}
\foreach \i/\j in {d/z_i}{
    \draw [-e+8,thick,LavenderMagenta] (\i) ..controls +($-0.5*(\i)+0.5*(\j)+(0,-0.8)$) .. (\j);
    \draw [-e+8,thick,LavenderMagenta] (\j) ..controls +($-0.5*(\j)+0.5*(\i)+(0,-0.8)$) .. (\i);
}
\foreach \i/\j in {y_i/b}{
    \draw [-e+8,thick,LavenderMagenta] (\i) ..controls +($-0.5*(\i)+0.5*(\j)+(0,0.8)$) .. (\j);
    \draw [-e+8,thick,LavenderMagenta] (\j) ..controls +($-0.5*(\j)+0.5*(\i)+(0,0.8)$) .. (\i);
}

\foreach \i in {b,c,d}{
    \node[dot] at (\i) [label=left:$\i$] {};
}
\foreach \i in {r,a,e}{
    \node[dot] at (\i) [label=right:$\i$] {};
}
\foreach \i in {x_i,y_i,z_i}{
    \node[dot,LavenderMagenta] at (\i) [label={[text=LavenderMagenta]right:$\i$}] {};
}

\end{tikzpicture}
        \caption{
        Let $B$ be the bidirected graph obtained from fives copies of the graph in the figure by identifying the black edges and black vertices. All edges are necessary to witness all signed connectivities of $B$.
        Furthermore,
        $\sum_{v \in \{a,b,c,d,e,f\}} \sum_{\alpha \in \{+,-\}} \lambda_B^{\alpha, \textrm{ path}}(r,v) = 11$ and $\sum_{v \in \{x_i,y_i,z_i\}} \sum_{\alpha \in \{+,-\}} \lambda_B^{\alpha, \textrm{ path}}(r,v) = 5$ for $i \in [5]$. There are seven black edges and for every $i \in [5]$ there exist six purple edges.
        }
        \label{fig:flame_edge-clean}
    \end{figure}

    \subsection{Ear-decomposition} \label{subsec:new_ear_decomposition}
    Let $B$ be a bidirected graph, let $(v, \alpha) \in \mathcal{V}(B)$ be a signed vertex.
        A subgraph $H$ of $B$ is called {\em $(v, \alpha)$-accessible} if for every signed vertex $(w, \beta) \in \mathcal{V}(H)$ there is a $(v,\alpha)$--$(w, \beta)$~path in $B$ only if there exists a $(v,\alpha)$--$(w, \beta)$~path in $H$.
        Note that every $(v, \alpha)$-accessible subgraph containing the endpoints of all $(v, \alpha)$-paths in $B$ is $(v, \alpha)$-spanning.
        
Given a bidirected graph $B'$ which is a subgraph of the bidirected graph $B$, a path or a $v$--$v$ almost path of $B$ is called an {\em ear} of $B'$  if it intersects $V(B')$ precisely in its first and last vertex.
Furthermore, we call an edge $e$ a \emph{bone} of $B'$ if precisely one endpoint of $e$ is in $V(B')$.
\begin{theorem}\label{thm:ear_decomposition}
    Let $B$ be a bidirected graph and let $(v, \alpha) \in \mathcal{V}(B)$ be a signed vertex such that for every $w \in V(B)$ there is a $(v, \alpha)$--$w$~path. Then there exists a sequence $B_0, \dots, B_n$ of $(v, \alpha)$-accessible subgraphs of $B$ such that
    \begin{itemize}
        \item $B_i$ is obtained from $B_{i-1}$ by adding an ear or a bone of $B_{i-1}$ for every $i \in [n]$;
        \item  $V(B_i) \supsetneq V(B_{i-1})$ unless $V(B_{i-1})=V(B)$ for every $i \in [n]$;
                \item $B_0= \{r\}$; and
        \item $B_n = B$.
    \end{itemize}
\end{theorem}
\begin{proof}
Let $B'$ be a maximal $(v,\alpha)$-accessible subgraph of $B$ for which there exists a sequence as desired.
Suppose for a contradiction that $B'$ is a proper subgraph of $B$, i.e.\ $E(B) \setminus E(B') \neq \emptyset$.
We can assume that $V(B') \neq V(B)$, since otherwise every edge of $E(B) \setminus E(B')$ is a $B'$-ear, contradicting the choice of $B'$.

Let $P$ be a $(v,\alpha)$-path in $B$ ending in some vertex of $V(B) \setminus V(B')$. Let $w$ be the first vertex of $P$ in $V(B) \setminus V(B')$ and let $\beta \in \{+,-\}$ such that $Pw$ ends in $(w, \beta)$.
Note that the last edge $e$ of $Pw$ is a bone of $B'$ and that there is a $(v, \alpha)$--$(w,\beta)$~path in $B' + e$ since $B'$ is $(v, \alpha)$-accessible.
If there is no $(v, \alpha)$--$(w,-\beta)$~path in $B$, then by adding the bone $e$ to $B'$ we obtain a $(v, \alpha)$-accessible graph, a contradiction to the choice of $B'$.

Otherwise, there exists a $(v, \alpha)$-$(w, -\beta)$~path $Q$ in $B$.
Let $x$ be its last vertex in $V(B')$.
Note that the concatenation $R$ of $xQ$ and $e$ is an ear of $B'$ with at least one internal vertex.
We show that $B' \cup R$ is $(v, \alpha)$-accessible, which contradicts the choice of $B'$.

Since $B'$ is $(v, \alpha)$-accessible it suffices to consider internal vertices of $R$. Let $y$ be some internal vertex of $R$ and let $\gamma \in \{+, - \}$ arbitrary. Since $B'$ is $(v, \alpha)$-accessible, there exists a (possibly trivial) $(v, \alpha)$-path of $B'$ that concatenates with either $Ry$ or $R^{-1}y$ to a $(v, \alpha)$-$(y,\gamma)$~path. This shows that $B' \cup R$ is accessible and completes the proof.
\end{proof}

We set $\sigma_B^{\alpha, \beta}(v,w) :=1$ if there is a $(v, \alpha)$--$(w, \beta)$~path in $B$ and otherwise $\sigma_B^{\alpha, \beta}(v,w) :=0$ for every signed vertex $(w, \beta) \in \mathcal{V}(B)$.

Given a bidirected graph $B$ and a signed vertex $(v, \alpha) \in \mathcal{V}(B)$ we say that a subgraph $H$ of $B$ is \emph{$(v,\alpha)$-spanning} if for every signed vertex $(w, \beta) \in \mathcal{V}(B)$ for which there is a $(v, \alpha)$--$(w, \beta)$~path in $B$, there exists a $(v, \alpha)$--$(w, \beta)$~path in $H$.

\begin{corollary}
\label{cor:edge-accessible-subgraph}
    Let $B$ be a bidirected graph and let $(v, \alpha) \in \mathcal{V}(B)$ be a signed vertex.
    Then there is a $(v,\alpha)$-spanning subgraph $B'$ of $B$ with $|E(B')| \leq \sum_{w \in V(B) \setminus \{v\}} \sum_{\beta \in \{+,-\}} \sigma_B^{\alpha, \beta}(v,w)$.
\end{corollary}
\begin{proof}
    Up to deletion of some vertices, we can assume that every vertex of $B$ is an endpoint of some $(v,\alpha)$-path.
    Let $B_0, \dots, B_n$ be as in~\cref{thm:ear_decomposition} and let $j $ be minimal with $V(B_j) = V(B)$. We show by induction on $i$ that $|E(B_i)| \leq \sum_{w \in V(B_i) \setminus \{v\}} \sum_{\beta \in \{+,-\}} \sigma_B^{\alpha, \beta}(v,w)$ for every $i\leq j$. Then $B_j$ is as desired.

    For $i=0$ this clearly holds because $|E(B_0)|=  0$. Let $i>0$. Since $V(B_i) \neq V(B_{i-1})$, $B_i$ is obtained from $B_{i-1}$ by adding an ear with at least one internal vertex or a bone. In both cases $|E(B_i) \setminus E(B_{i-1})| \leq \sum_{w \in V(B_i) \setminus V(B_{i-1})} \sum_{\beta \in \{+,-\}} \sigma_B^{\alpha, \beta}(v,w)$.
\end{proof}
        
\subsection{Existence of flames} \label{subsec:existence_edge_flames}

\EdgeFlameBidirected*
\begin{proof}
    Let $D$ be the trail-skeleton of $B$ and $C_1, \hdots, C_k$ be trail-undirectable components of $B$.
    Let $S_i$ be a $(c_i, \alpha_i)$-spanning subgraph of $B[C_i]$ given by~\cref{cor:edge-accessible-subgraph} for each $i \in \{1, \hdots, k\}$.
    Let $\hat{F}$ be the rooted flame for $D$ given by~\cref{t: LovFlame}. We claim that there exists a mapping $m: E(D)\to E(B)$, such that the subgraph $F$ of $B$ induced by $m(E(\hat{F})) \cup \bigcup_{i=1}^k E(S_i)$ is a desired flame.

    We prove that $\lambda_F^{\beta, \textrm{ path}}(r,v)=\lambda_B^{\beta, \textrm{ path}}(r,v)$ for every $v \in V(B) \setminus \{r\}$ and every $\beta \in \{+,-\}$.
    Let $v \in V(B) \setminus \{r\}$ be arbitrary. Assume first that $v \in V(D)$. Then $\lambda_D(r, v) = \lambda_B^{+, \textrm{ path}}(r, v) + \lambda_B^{-, \textrm{ path}}(r, v)$ by \Cref{lem:disjoint_paths_skeleton,lem:trail-solid-one-sign}.
    Let $\mathcal{P}$ be a set of $\lambda_D(r, v)$ edge disjoint $r$--$v$~paths in $\hat{F}$, which exists by~\cref{t: LovFlame}. 
    Now, by \Cref{cor:trail-skeleton-edge-disjoint,lem:paths_in_path_reachable}, there is a set $\mathcal{P}'$ of $\lambda_D(r, v)$ edge disjoint $r$--$v$~paths in $F$. 
    Since at most $\lambda_B^{+, \textrm{ path}}(r, v)$ paths of $\mathcal{P}'$ are $r$--$(v, +)$ paths and at most $\lambda_B^{-, \textrm{ path}}(r, v)$ paths of $\mathcal{P}'$ are $r$--$(v, -)$ paths, the set $\mathcal{P}'$ witnesses that $\lambda_{F}^{\beta, \textrm{ path}}(r, v) = \lambda_{B}^{\beta, \textrm{ path}}(r, v)$ for $ \beta \in \{+, -\}$. 
    
    Next assume that $v \in V(B) \setminus V(D)$ and let $\beta \in \{+,-\}$ be arbitrary. By the construction of the trail-skeleton $D$, there is $i \in [k]$ with $v \in V(C_i)$ and $\lambda_B^{\beta, \textrm{ path}}(r, v) \leq 1$ by~\cref{lem:disjoint_paths_skeleton_inside}. 
    If $\lambda_B^{\beta, \textrm{ path}}(r, v) = 0$, we are done.
    Thus we can assume $\lambda_B^{\beta, \textrm{ path}}(r, v) = 1$ and let $P$ be an $r$--$(v, \beta)$~path. By \cref{lem:component_properties}, $P$ contains $c_i$ and the subpath $c_iP$ starts in $(c_i, \alpha_i)$ and is contained in $C_i$.
    This implies that $\lambda_B^{-\alpha_i, \textrm{ path}}(r,c_i) \geq 1$ and there is a $(c_i, \alpha_i)$--$(v,\beta)$~path in $C_i$.
    By choice of $\hat{F}$ and by~\cref{lem:disjoint_paths_skeleton_inside}, there is $r$--$c_i$~path $Q$ in $F$.
    Furthermore, by~\cref{lem:component_properties}, $Q$ ends in $(c_i,-\alpha_i)$ and intersects $V(C_i)$ only in its endpoint.
    Since the subgraph $S_i$ of $C_i$ is $(c_i, \alpha_i)$-spanning, there is a $(c_i, \alpha_i)$--$(v, \beta)$~path $Q'$ in $S_i$.
    Then the concatenation of $Q$ and $Q'$ is the desired path in $F$.
    This shows that $\lambda_F^\beta(r,v)=\lambda_B^\beta(r,v)$ for every $v \in V(B) \setminus \{r\}$ and every $\beta \in \{+,-\}$.

    Finally, we show that $F$ is a flame.
    By the choice of $S_i$ and by~\cref{lem:component_properties}, we have
    \[
    |E(S_i)| \leq \sum_{v \in V(C_i) \setminus V(D)} \sum_{\beta \in \{+,-\}} \sigma_{C_i}^{\alpha_i, \beta}(c_i, v) \leq \sum_{v \in V(C_i) \setminus V(D)} \sum_{\beta \in \{+,-\}} \lambda_B^{\beta, \textrm{ path}}(r,v).
    \]
    Since each vertex of $V(B) \setminus V(D)$ is contained in precisely one $C_i$, we can derive
    \[
    \left| E\left( \bigcup_{i=1}^{k} S_i\right)  \right|  \leq \sum_{v \in V(B) \setminus V(D)} \sum_{\beta \in \{+,-\}} \lambda_B^{\beta, \textrm{ path}}(r,v).
    \]
    The choice of $\hat{F}$ ensures that $|E(\hat{F})| \leq \sum_{v \in V(D) \setminus \{r\}} \lambda_D(r,v)$, which implies by~\Cref{lem:disjoint_paths_skeleton,lem:trail-solid-one-sign} that
    \[
    |E(\hat{F})| \leq \sum_{v \in V(D) \setminus \{r\}} \sum_{\beta \in \{+, -\}} \lambda_B^{\beta, \textrm{ path}}(r,v).
    \] 
    This proves \[|E(F)| \leq \sum_{v \in V(B) \setminus \{r\}} \sum_{\beta \in \{+,-\}} \lambda_F^{\beta, \textrm{ path}}(r,v). \]
    Thus $F$ is a flame, which completes the proof.
\end{proof}

\section{Application: Pym's theorem} \label{sec:edge-pym}
Pym's theorem \cite{pym_original}, also called the linkage theorem, is another classical result related to connectivity in graphs. Originally, it was formulated in terms of vertex-connectivity:

\begin{theorem}[\cite{pym_original}]\label{thm:pym_directed}
    Let $D$ be a directed graph, let $X$ and $Y$ be sets of vertices, and let $\mathcal{P}$ and $\mathcal{Q}$ be sets of vertex-disjoint paths in $D$ that start in $X$ and end in $Y$.
    Then there is a set $\mathcal{R}$ of vertex-disjoint $X$--$Y$ paths in $D$ such that the set of the first vertices of $\mathcal{P}$ is included in the set of the first vertices of $\mathcal{R}$ and the set of the last vertices of $\mathcal{Q}$ is included in the set of the last vertices of $\mathcal{R}$.
\end{theorem}
\noindent
By applying \cref{thm:pym_directed} to the directed line graph one obtains its edge-variant:
\begin{theorem}\label{thm:edge-pym_directed}
    Let $D$ be a directed graph rooted at $r$, let $x \in V(D) \setminus \{r\}$ and let $\mathcal{P}$ and $\mathcal{Q}$ be sets of edge-disjoint $r$--$x$~paths in $D$.
    Then there is a set $\mathcal{R}$ of edge-disjoint $r$--$x$ paths in $D$ covering the first edges of the paths in $\mathcal{P}$ and the last edges of the paths in $\mathcal{Q}$.
\end{theorem}

The verbatim transfer of \cref{thm:edge-pym_directed} to bidirected graphs fails (see~\cref{fig:pym}). In this section, we use our understanding of bidirected graphs to generalize~\cref{thm:edge-pym_directed} to edge-clean bidirected graphs:

\begin{figure}
    \centering
    \begin{tikzpicture}
\coordinate (r) at (0,0.5);
\coordinate (a) at (1.5,0);
\coordinate (b) at (1.5,1);
\coordinate (x) at (3,0.5);

\draw[line width=9pt, LavenderMagenta, line cap=round] (r) to ($(a)+(0.05,0)$);
\draw[line width=9pt, LavenderMagenta, line cap=round] (x) to ($(b)+(0.05,0)$);
\draw[line width=9pt, LavenderMagenta, line cap=round] ($(a)+(0.05,0)$) to ($(b)+(0.05,0)$);

\draw[line width=9pt, PastelOrange, line cap=round] (r) to ($(b)+(-0.05,0)$);
\draw[line width=9pt, PastelOrange, line cap=round] (x) to ($(a)+(-0.05,0)$);
\draw[line width=9pt, PastelOrange, line cap=round] ($(a)+(-0.05,0)$) to ($(b)+(-0.05,0)$);

\foreach \i/\j in {a/b,b/a}{
    \draw[-e+8,thick] (\i) to (\j);
}
\foreach \i/\j in {x/a,x/b,a/r,b/r}{
    \draw[thick] (\i) to (\j);
}

\foreach \i in {r,x}{
    \node[dot] at (\i) [label=below:$\i$] {};
}
\foreach \i in {a,b}{
    \node[dot] at (\i) [] {};
}
\node[] at ($0.5*(r)+0.5*(a)+(0,-0.4)$) [] {$e$};
\node[] at ($0.5*(a)+0.5*(x)+(0,-0.4)$) [] {$f$};
\end{tikzpicture}
    \caption{A rooted bidirected graph $B$ in which there is an $r$--$x$~paths that starts in $e$ and one that ends in $f$. There are no two edge-disjoint $r$--$x$~paths in $B$ and no $r$--$x$~path in $B$ contains $e$ and $f$.}
    \label{fig:pym}
\end{figure}
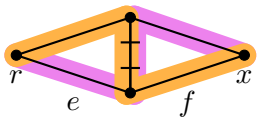

        \begin{theorem}\label{thm:edge-pym-bidirected}
            Let $B$ be an edge-clean bidirected graph rooted at $r$, let $x \in V(B) \setminus \{r\}$ and
            let $\mathcal{P}$ and $\mathcal{Q}$ be sets of edge-disjoint $r$--$x$~paths in $B$. Then there is a set $\mathcal{R}$ of edge-disjoint $r$--$x$ paths in $B$ covering the first edges of the paths in $\mathcal{P}$ and the last edges of the paths in $\mathcal{Q}$.
        \end{theorem}

        \begin{proof}
            Let $D$ be the trail-skeleton of $B$, and $f$ be a bijection from $E(D)$ to a subset $F\subseteq E(B)$.
            If $x \in V(B) \setminus V(D)$, then by the construction of $D$, $x$ is contained in some trail-undirectable component of $B$, which means $\lambda_B^{\textrm{path}}(r,x) = 1$ by \cref{lem:disjoint_paths_skeleton_inside}. It implies that $|\mathcal{P}|=|\mathcal{Q}|=1$ and provides the desired set $\mathcal{R}$ by \cref{lem:component_properties}. 
            Thus we can assume that $x \in V(D)$. The definition of edge-cleanness and \cref{lem:component_properties} imply that, for every $r$--$x$ path $P$ in $B$, the first and last edge of $P$ is in $F$.
            By \Cref{lem:bijection_trail_to_trail_skeleton1,rem:trail-skeleton-edge-disjoint}, there exist sets $\Tilde{\mathcal{P}}, \Tilde{\mathcal{Q}}$ of edge-disjoint $r$--$x$~paths in $D$ such that $f$ induces a bijection between the first edges of $\mathcal{P}$ and the first edges of $\Tilde{\mathcal{P}}$, and between the last edges of $\mathcal{\mathcal{Q}}$ and the last edges of $\Tilde{\mathcal{Q}}$. 

            By~\cref{thm:edge-pym_directed}, there exists a set $\Tilde{\mathcal{R}}$ of edge-disjoint $r$--$x$~paths in $D$ such that $\Tilde{\mathcal{R}}$ covers the first edges of $\Tilde{\mathcal{P}}$ and the last edges of $\Tilde{\mathcal{Q}}$. It provides the desired set $\mathcal{R}$ by~\Cref{lem:bijection_trail_to_trail_skeleton2,lem:paths_in_path_reachable,cor:trail-skeleton-edge-disjoint}.
        \end{proof}

\section{From edge-disjoint to vertex-disjoint}\label{sec: from edge to vertex}
In this section we generalize the results of~\Cref{sec:edge-menger,sec:edge-flame,sec:edge-pym} to (internally) vertex-disjoint paths in \emph{clean} bidirected graphs.
The main tool of this section is an auxiliary graph, which we define in~\cref{sec:aux-graph}.

Recall that a trail $T$ is an \emph{almost path} if by removing its last edge we obtain a path.
An $r$-rooted bidirected graph $B$ is \emph{clean} if there does not exist a nontrivial $r$--$r$~almost path.
Given a bidirected graph $B$ rooted at $r$, we call a vertex $v$ of $B$ {\em plain} if there exists $\alpha \in \{+,-\}$ such that there is no $r$--$(v, \alpha)$~path in $B$.
Let $\Pi_B \subseteq V(B)$ be the set of all plain vertices of $B$.
Given vertices $v,w \in V(B)$, let $\kappa_B(v,w)$ be the maximum number of internally vertex-disjoint $v$--$w$~paths. Similarly, we define $\kappa_B^\alpha(v,w)$ for some $\alpha \in \{+,-\}$.

\begin{lemma}\label{lem:non_plain_vertex}
    Let $B$ be a clean bidirected graph rooted at $r$. Then for every non-plain vertex $v$ of $B$ we have $\kappa_B(r,v)=1$.
\end{lemma}
\begin{proof}
Suppose not for a contradiction. We must have $\kappa_B^+(r,v) \geq 2$ or $\kappa_B^-(r,v) \geq 2$, since otherwise there are two internally disjoint $r$--$v$~paths that end in different signs at $v$ and thus their union is a nontrivial $r$--$r$~almost path. Since $v$ is not plain, up to symmetry of $+$ and $-$, we have $\kappa_B^+(r,v) \geq 2$ and $\kappa_B^-(r,v) \geq 1$. Let $P$ and $Q$ be two internally vertex-disjoint $r$--$(v,+)$~paths. Furthermore, let $R$ be some $r$--$(v,-)$~path with minimum number of edges in $E(B) \setminus (E(P) \cup E(Q))$. If $R$ is internally vertex-disjoint to either $P$ or $Q$, then either $Rv P^{-1}$ or $Rv Q^{-1}$ is a nontrivial $r$--$r$~almost path, contradicting that $B$ is clean.

Thus we can assume that $R$ contains internal vertices of $P$ and $Q$. Further we can assume by symmetry that the first internal vertex of $R$ in $V(P) \cup V(Q)$ is contained in $V(P)$. Let $x$ be the first internal vertex of $Q$ in $V(R)$. Then $Rx$ contains an edge of $E(B) \setminus (E(P) \cup E(Q))$. Either $QxR^{-1}$ is an $r$--$r$~almost path or $QxR$ is a path. Note that the former case contradicts that $B$ is clean and the latter case contradicts the choice of $R$ since $QxR$ has fewer edges in $E(B) \setminus (E(P) \cup E(Q))$ than $R$, giving the desired contradiction.
\end{proof}

\subsection{Auxiliary graph} \label{sec:aux-graph}

Now, we construct an auxiliary bidirected graph $a(B)$ from $B$ as follows (see~\cref{fig:aux-graph}). For every vertex $v$ in $\Pi_B$, we replace $v$ with two vertices $v^+$ and $v^-$ and an edge between $v^+$ and $v^-$ with sign $+$ at $v^-$ and sign $-$ at $v^+$. Furthermore, every edge $e$ incident with $v$ with sign $+$ in $B$ becomes an edge incident with $v^+$ with sign $+$ in $a(B)$, and similarly every edge $e$ incident with $v$ with sign $-$ in $B$ becomes an edge incident with $v^-$ with sign $-$. Observe that $V(a(B)) = (V(B) \setminus \Pi_B) \cup \{v^+, v^- \mid v \in \Pi_B\}$ and $B$ is obtained from $a(B)$ by contracting every edge $v^+v^-$ with $v \in \Pi_B$.

\begin{figure}
\centering
\begin{subfigure}[b]{0.44\textwidth}
\centering
\begin{tikzpicture}
\coordinate (r) at (0,1.4);
\coordinate (a) at (1.4,2.8);
\coordinate (b) at (2.8,2.8);
\coordinate (c) at (4.2,2.8);
\coordinate (d) at (5.6,2.8);
\coordinate (e) at (2.8,1.4);
\coordinate (f) at (4.2,1.4);
\coordinate (g) at (5.6,1.4);
\coordinate (h) at (1.4,0);
\coordinate (i) at (2.8,0);
\coordinate (j) at (4.2,0);

\foreach \i in {a,c,d,f,h,j}{
    \draw[line width=9pt, PastelOrange, line cap=round] (\i) to (\i);
}
\foreach \i in {a,b,c,d,e}{
    \node[dot] at (\i) [label=above:$\i$] {};
}
\foreach \i in {f}{
    \node[dot] at (\i) [label=below left:$\i$] {};
}
\foreach \i in {h,i,j,g,r}{
    \node[dot] at (\i) [label=below:$\i$] {};
}
\foreach \i/\j in {r/a, h/e, e/j, f/j, d/g}{
    \draw[thick] (\i) to (\j);
}
\foreach \i/\j in {r/h, h/i, i/e, j/i, c/f, b/c, c/b, b/f, f/b}{
    \draw[-e+8,thick] (\i) to (\j);
}
\foreach \i/\j in {a/b, b/a, f/g}{
     \draw [-e+8,thick] (\i) ..controls +($-0.5*(\i)+0.5*(\j)+(0,-0.4)$) .. (\j);
}
\foreach \i/\j in {b/a}{
    \draw [-e+8,thick] (\i) ..controls +($-0.5*(\i)+0.5*(\j)+(0,+0.4)$) .. (\j);
}
\foreach \i/\j in {f/g}{
    \draw [thick] (\i) ..controls +($-0.5*(\i)+0.5*(\j)+(0,+0.4)$) .. (\j);
}

\end{tikzpicture}
\caption{The bidirected graph $B$ rooted in $r$. The vertices of $\Pi_B$ are marked orange.}
\end{subfigure}
\hfill
\begin{subfigure}[b]{0.55\textwidth}
\centering
\begin{tikzpicture}
\coordinate (r) at (0,1.4);
\coordinate (a^-) at (1.4,2.8);
\coordinate (a^+) at (2.1,2.8);
\coordinate (b) at (3.5,2.8);
\coordinate (c^+) at (4.9,2.8);
\coordinate (c^-) at (5.6,2.8);
\coordinate (d^-) at (7,2.8);
\coordinate (d^+) at (6.3,2.8);
\coordinate (e) at (3.5,1.4);
\coordinate (f^+) at (4.9,1.4);
\coordinate (f^-) at (5.6,1.4);
\coordinate (g) at (7,1.4);
\coordinate (h^+) at (1.4,0);
\coordinate (h^-) at (2.1,0);
\coordinate (i) at (3.5,0);
\coordinate (j^-) at (4.9,0);
\coordinate (j^+) at (5.6,0);

\foreach \i in {a,c,d,f,h,j}{
     \draw[line width=9pt, PastelOrange, line cap=round] (\i^+) to (\i^-);
}
\foreach \i in {a^+,a^-,b,c^-,c^+,d^-,d^+,e,f^-}{
    \node[dot] at (\i) [label=above:$\i$] {};
}
\foreach \i in {r,h^-,h^+,i,j^-,j^+,g,f^+}{
    \node[dot] at (\i) [label=below:$\i$] {};
}
\foreach \i in {a,c,d,f,h,j}{
    \draw[-e+4,thick] (\i^+) to (\i^-);
}
\foreach \i/\j in {r/a^-, h^-/e, e/j^-, f^-/j^-, d^-/g}{
    \draw[thick] (\i) to (\j);
}
\foreach \i/\j in {r/h^+, h^-/i, i/e, j^-/i, c^-/f^+, b/c^+, c^+/b, b/f^+, f^+/b}{
    \draw[-e+8,thick] (\i) to (\j);
}
\foreach \i/\j in {a^+/b, b/a^+, f^-/g}{
     \draw [-e+8,thick] (\i) ..controls +($-0.5*(\i)+0.5*(\j)+(0,-0.4)$) .. (\j);
}
\foreach \i/\j in {b/a^+}{
    \draw [-e+8,thick] (\i) ..controls +($-0.5*(\i)+0.5*(\j)+(0,+0.4)$) .. (\j);
}
\foreach \i/\j in {f^-/g}{
    \draw [thick] (\i) ..controls +($-0.5*(\i)+0.5*(\j)+(0,+0.4)$) .. (\j);
}
\end{tikzpicture}
\caption{The auxiliary graph $a(B)$. The auxiliary edges are marked orange.}
\end{subfigure}
\caption{A rooted bidirected graph $B$ and its auxiliary graph $a(B)$.}
\label{fig:aux-graph}
\end{figure}
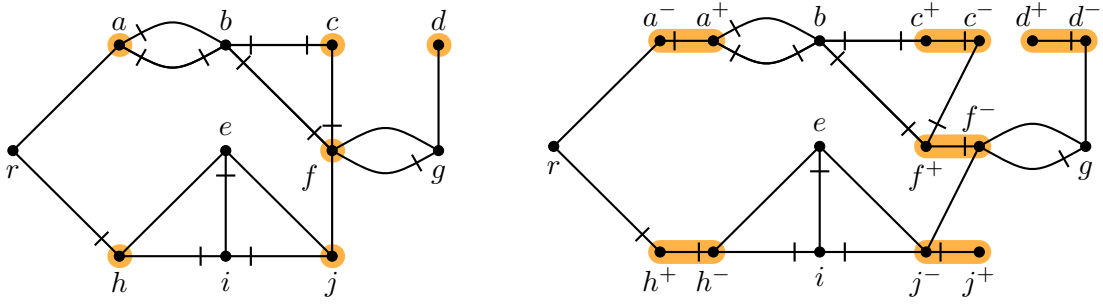

%

Call edges of $a(B)$ whose ends are $v^+$ and $v^-$ for $v \in \Pi_B$ {\em auxiliary edges}.  Next, we define a mapping $g_B$ between $r$-trails of $a(B)$ and $r$-trails of $B$, where for an $r$-trail $T$ of $a(B)$, the image $g_B(T)$ is the $r$-trail of $B$ formed from $T$ by contracting auxiliary edges.  

A path $P$ of $a(B)$ is called a {\em proper path} if the last edge of $P$ is not an auxiliary edge. Let $\mathcal{P}(a(B))$ be the set of proper $r$-paths of $a(B)$ and let $\mathcal{P}(B)$ be the set of $r$-paths of $B$.

\begin{lemma}\label{lem:aux_path_bijection}
    Let $B$ be a bidirected graph rooted at $r$ and let $a(B)$ be its auxiliary graph. Then $g_B$ is a bijection from $\mathcal{P}(a(B))$ to $\mathcal{P}(B)$. 
\end{lemma}
\begin{proof}
First we show that $g_B$ is injective. Let $P$ and $Q$ be two distinct proper $r$-paths of $a(B)$. Since $P$ and $Q$ are distinct, up to symmetry between $P$ and $Q$, there is a (possibly trivial) path $P'$ and an edge $\ve_1$ such that $P'$ is an initial segment of $Q$, and $P' \circ \ve_1$ is an initial segment of $P$, and $e_1$ does not appear in $Q$. 

Suppose $Q = P'$. Let $e_2$ be the last edge of $P$. Since $P$ is proper, $e_2$ is not an auxiliary edge, so $e_2$ appears as an edge in the path $g_B(P)$ of $B$. However, $e_2$ is not an edge of $P'$, so $e_2$ does not appear in the path $g_B(Q)$. 

Next, suppose $Q \neq P'$. Then there is an edge $e_2$ distinct from $e_1$ such that $P' + e_2$ is an initial segment of $Q$. Since each vertex of $a(B)$ is incident to at most one auxiliary edge by construction, we may assume up to symmetry between $e_1$ and $e_2$ that $e_1$ is not an auxiliary edge. But now $e_1$ appears as an edge in the path $g_B(P)$ of $B$ and not in the path $g_B(Q)$ of $B$. This completes the proof that $g_B$ is injective. 

Next we prove that $g_B$ is surjective. Let $P$ be an $r$-path of $B$. Let $P'$ be the path formed by replacing $v$ with $v^+v^-$ or $v^-v^+$ for each vertex $v \in V(P) \cap \Pi_B$ depending on the sign in which $P'$ enters $v$, then removing the final edge if it is an auxiliary edge. Now $P'$ is an $r$-path of $a(B)$ and $g_B(P') = P$. This completes the proof. 
\end{proof}

\begin{corollary}\label{cor:path-reachable}
    Let $B$ be a bidirected graph rooted at $r$ and let $e \in E(B)$. Then there is an $r$--$\ve$~path in $B$ if and only if there is an $r$--$\ve$~path in $a(B)$.
    Furthermore, $B$ is path-reachable and the vertices of $\Pi_B$ are non-isolated if and only if $a(B)$ is path-reachable.
\end{corollary}
\begin{proof}
By~\cref{lem:aux_path_bijection}, for every $e \in E(B)$ there is an $r$--$\ve$~path in $B$ if and only if there is an $r$--$\ve$~path in $a(B)$.
Note that if an auxiliary edge $v^+v^-$ is path-reachable in $a(B)$, then there is an $r$--$v$~path in $B$ and in particular, $v$ is not isolated.
Thus it suffices to show that if $B$ is path-reachable and $v \in \Pi_B$ is not isolated, then the auxiliary edge $v^+v^-$ is path-reachable.
Let $e$ be an edge incident with $v$. Then there is an $r$--$e$~path $P$ in $a(B)$ by~\cref{lem:aux_path_bijection}. By construction of $a(B)$, either $P$ contains an $r$--$v^+v^-$~path or the concatenation of $P$ and $v^+v^-$ is an $r$--$v^+v^-$~path. Thus $v^+v^-$ is path-reachable.
\end{proof}

\begin{lemma}\label{lem:aux-graph-trail-to-almost path}
    Let $B$ be a bidirected graph rooted at $r$ and let $\ve$ be some orientation of an edge in $a(B)$. If there is an $r$--$\ve$~trail in $a(B)$, then there is an $r$--$\ve$~almost path in $a(B)$. 
\end{lemma}

\begin{proof}
    Let $u$ and $v$ be tail and head of $\ve$, respectively, and let $-\alpha$ be the sign of $e$ at $u$. Since there is an $r$--$\ve$~trail in $a(B)$, there is an $r$--$(u, \alpha)$~trail in $a(B) - e$. 
    
    First we show that there is an $r$--$(u, \alpha)$~path in $a(B)$. Let $T$ be an $r$--$(u, \alpha)$~trail of $a(B) - e$ that minimizes the number of those vertices that are used at least twice. If $T$ has no vertex used at least twice, then $T$ is an $r$--$(u, \alpha)$~path in $a(B)$, so assume that some vertex is used at least twice in $T$.

    Let $w$ be the last vertex of $T$ that is used twice. Let $S$ be the terminal segment of $T$ that starts at the last appearance of $w$ and let $R$ be the initial segment of $T$ that terminates at the first appearance of $w$ in $T$, so $S$ is in particular a path and $R$ is a trail that meets $w$ only in its last vertex. Furthermore, $R$ and $S$ are disjoint except for $w$. Up to symmetry between $+$ and $-$, we may assume that $R$ is an $r$--$(w, +)$~trail and $S$ is a $(w, +)$--$(u, \alpha)$~trail, since otherwise $S \cup R$ is an $r$--$(u, \alpha)$~trail that uses $w$ only once. 

    \begin{claim}
        There is a $(u, \alpha)$--$(w, -)$~path in $a(B)$.
    \end{claim}
    \begin{claimproof}
    Since $w$ occurs at least twice as an internal vertex of the trail $T \circ \ve$ in $a(B)$, it follows that $w$ is not incident to an auxiliary edge and thus $w \in V(B) \setminus (\Pi_B \cup \{r\})$. Therefore there is an $r$--$(w, -)$~path in $B$. By~\cref{lem:aux_path_bijection}, there is an $r$--$(w, -)$~path $P$ in $a(B)$. Let $a$ be the last vertex of $S$ contained in $P$ (possibly $a = w$). If $Pa \cup aS$ is a path, then $Pa \cup aS$ is an $r$--$(u, \alpha)$~path that contradicts the choice of $T$. Therefore, $aS \cup aP$ is a $(u, \alpha)$--$(w, -)$~path. This proves the claim. 
    \end{claimproof}

    By the claim above, we may choose a $(u, \alpha)$--$(w, -)$~path $Q$ of $a(B)$ that minimizes $E(Q) \setminus E(S)$. Let $x$ be the first vertex of $R$ that meets $Q$ (possibly $x = w$) and let $y$ be the last vertex of $S$ that meets $xQ$. If $Rx \cup x Q^{-1}$ is a path, then it is an $r$--$(u, \alpha)$~trail that contradicts the choice of $T$, since $w$ appears at most once in $Rx \cup x Q^{-1}$. Therefore, $Rx \cup xQy$ is a trail which meets $w$ at most once. Now, if $Rx \cup xQy \cup yS$ is a trail, then it contradicts the choice of $T$, since $w$ appears at most once. Therefore, $S^{-1}y \cup yQ$ is a $(u, \alpha)$--$(w, -)$~path that contradicts the choice of $Q$.  

    We have now shown that $T$ is a $r$--$(u, \alpha)$~path of $a(B)$. Therefore, $T \circ \ve$ is an $r$--$\ve$~almost path of $a(B)$, which completes the proof.
\end{proof}

We can use this to relate trail-reachability in $a(B)$ with a kind of reachability in $B$. More precisely, we say a bidirected graph $B$ rooted at $r$ is {\em reachable} if $B$ contains an $r$--$e$~almost path for every $e \in E(B)$.
Similarly, an edge $e \in E(B)$ is {\em directable} if there is a unique orientation $\ve$ of $e$ such that $B$ contains an $r$--$\ve$~almost path; we call this orientation $\ve$ the {\em natural orientation} of $e$.

\begin{corollary}
\label{cor:reach}
Let $B$ be a rooted bidirected graph. Then: 
\begin{enumerate}[label=(\alph*)]
    \item\label{itm:reach_2} for all $e \in E(B)$, there is an $r$--$\ve$~almost path in $B$ if and only if there is an $r$--$\ve$~trail in $a(B)$;
    \item\label{itm:reach_3} for all $e \in E(B)$, $e$ is directable in $B$ if and only if $e$ is trail-directable in $a(B)$;
    \item\label{itm:reach_1}  $B$ is reachable if and only if $a(B)$ is trail-reachable; and
    \item\label{itm:reach_4} $B$ is clean if and only if $a(B)$ is edge-clean.
\end{enumerate} 
\end{corollary}

Now, we investigate the relation of internally vertex-disjoint paths in $B$ and edge-disjoint paths in $a(B)$ under the condition that $B$ is clean.
For every $x \in V(B)$, let $\mathbf{a}(x) := \{x\}$ if $x \not \in \Pi_B$ and $\mathbf{a}(x) := \{x^+, x^-\}$ if $x \in \Pi_B$. 

\begin{theorem}\label{thm:aux_disjoint_paths}
    Let $B$ be a clean bidirected graph rooted at $r$ and let $a(B)$ be its auxiliary graph. Then, for every $x \in V(B) \setminus \{r\}$, a collection $\mathcal{P}$ of $r$--$\mathbf{a}(x)$~proper paths of $a(B)$ is edge-disjoint if and only if the corresponding collection of $r$--$x$~paths of $B$ given by $\mathcal{P}' \coloneqq\{g_B(P) \mid P \in \mathcal{P}\}$ is internally vertex-disjoint. 
\end{theorem}

\begin{proof}
    First, assume that $\mathcal{P}$ is edge-disjoint and suppose for a contradiction that $\mathcal{P}'$ is not internally vertex-disjoint. Then there are two paths $P'$ and $Q'$ of $\mathcal{P}'$ which have an internal vertex in common; let $v$ be the first vertex of $P'$ that is in $Q'$. Now, the paths $P'v$ and $Q'v$ are internally vertex-disjoint.
    Thus there is a sign $\alpha \in \{+,-\}$ such that the paths $P'$ and $Q'$ arrive in $v$ with sign $\alpha$ since otherwise $P'v(Q')^{-1}$ is a nontrivial $r$--$r$~almost path, contradicting that $B$ is clean. 
    Then $\kappa_B^\alpha(r, v) \geq 2$ and the vertex $v$ is plain by~\cref{lem:non_plain_vertex}.
    Let $P = g_B^{-1}(P')$ and $Q = g_B^{-1}(Q')$ be the proper $r$--$\mathbf{a}(x)$~paths of $\mathcal{P}$ corresponding to $P'$ and $Q'$. Since $v$ is plain, the edge $v^+v^-$ appears in both $P$ and $Q$, contradicting that $\mathcal{P}$ is edge-disjoint.

    Conversely, assume that $\mathcal{P}'$ is internally vertex-disjoint. Since the paths of $\mathcal{P}$ are formed from the paths of $\mathcal{P}'$ by adding auxiliary edges, $\mathcal{P}'$ is edge-disjoint, except for possibly auxiliary edges. But if two proper paths $P'$ and $Q'$ of $\mathcal{P}'$ share an auxiliary edge $v^+v^-$, then the corresponding paths $g_B(P')$ and $g_B(Q')$ of $\mathcal{P}$ both contain $v$ as an internal vertex, a contradiction. 
\end{proof}
    \noindent
    By the definition of $g_B$ we can deduce:
    
    \begin{corollary}\label{cor:aux_connectivity}
        Let $B$ be a clean bidirected graph rooted at $r$. Then for every $v \in \Pi_B$
        \begin{enumerate}[label=(\roman*)]
            \item $\max_{\alpha \in \{+,-\}} \lambda_{a(B)}^{\textrm{path}}(r,v^\alpha) = \kappa_B(r,v)$;
            \item $\lambda_{a(B)}^{\alpha, \textrm{ path}}(r,v^\alpha) = \kappa_B^\alpha(r,v)$ for $\alpha \in \{+,-\}$; and
            \item $\lambda_{a(B)}^{-\alpha, \textrm{ path}}(r,v^\alpha) = \min\{1,\kappa_B^{-\alpha}(r,v)\}$  for $\alpha \in \{+,-\}$.
        \end{enumerate}
        Furthermore, for every $v \in V(B) \setminus (\Pi_B \cup \{r\})$
        \begin{enumerate}[label=(\roman*), start=4]
           \item $\lambda_{a(B)}^{\textrm{path}}(r,v) = \kappa_B(r,v)$; and
            \item $\lambda_{a(B)}^{\alpha, \textrm{ path}}(r,v) = \kappa_B^\alpha(r,v)$ for $\alpha \in \{+,-\}$.
        \end{enumerate}
    \end{corollary}

\subsection{A strong Menger's theorem for internally vertex-disjoint paths}
    We prove the following variant of~\cref{thm:strong_menger_path} for internally vertex-disjoint paths:
    \begin{theorem} \label{thm:strong_menger_vertex}
    Let $B$ be a clean bidirected graph rooted at $r$ and let $x \in V(B) \setminus \{r\}$ such that there is no edge between $r$ and $x$. Then 
    \[ \kappa_B(r, x) = \min \left\lbrace  \left| \epsilon_{r,B} (X)
            \right|\ :\ 
             \{x\} \cup \InVertices_B(x) \subseteq X 
             \subseteq V(B)\setminus \{ r \}\right\rbrace. \]
    \end{theorem}
    \noindent
    Here, $\InVertices_B(x)$ is the set $\{v \in V(B) \setminus \{r,x\}: \text{there is an $r$--$x$~path with second to last vertex $v$}\}$.
Furthermore, we define $\epsilon_{r,B} (X )$ to be the set of vertices $v \in X$ for which there exists an orientation $\ve$ of an edge of $B$ such that $v$ is the head of $\ve$, the tail of $\ve$ is not in $X$ and there exists an $r$--$\ve$~path in $B$.
Note that~\cref{thm:strong_menger_vertex} fails in general for non-clean bidirected graphs~\cite{bowler2023mengertheorembidirected}*{Figure 1}.


\begin{proof}
    Since there is no edge between $r$ and $x$, every $r$--$x$~path has length at least $2$ and thus has an internal vertex in every such set $\epsilon_{r,B} (X)$. Thus
    \[
    \kappa_B(r, x) \leq \min \left\lbrace  \left| \epsilon_{r,B} (X)
            \right|\ :\ 
             \{x\} \cup \InVertices_B(x) \subseteq X 
             \subseteq V(B)\setminus \{ r \}\right\rbrace.
    \]
    
    To prove the inequality ``$\geq$'', we apply~\cref{thm:strong_menger_path}.
    If $x \in \Pi_B$, then there exists $\alpha \in \{+,-\}$ such that $\kappa_B^{-\alpha}(r,x) = 0$, and we set $x^*:= x^\alpha$.
    Otherwise, we set $x^*:= x$.
    Note that $\kappa_B(r,x) = \lambda_{a(B)}^{\textrm{path}}(r,x^*)$ by~\cref{thm:aux_disjoint_paths} and by the choice of $x^*$.
    
    By~\cref{cor:reach}, the auxiliary graph $a(B)$ is edge-clean. 
    Thus we can apply~\cref{thm:strong_menger_path} to $a(B)$ and obtain a set $Y \subseteq V(a(B)) \setminus \{r\}$ with $x^* \in Y$ such that $
    \lambda_B^{\textrm{path}}(r, x^*) = \left| \delta_{r,a(B)}^{\mathsf{path}} (Y )
            \right|$.
    We prove that the set $X:= \{v \in V(B): \mathbf{a}(v) \cap Y \neq \emptyset\} \cup \InVertices_{B}(x)$ witnesses the remaining inequality.
    Note that $r \notin X$ and $\{x\} \cup \InVertices_B(x) \subseteq X$ by the construction of $X$.
    We have to show that $\kappa_B(r,x) \geq |\epsilon_{r,B} (X)|$.

    We define an injective function $\Phi$ from $\epsilon_{r,B} (X)$ to $\delta_{r,a(B)}^{\mathsf{path}} (Y )$.
    Let $F$ be the set of oriented edges in $\delta_{r,a(B)}^{\mathsf{path}} (Y )$ whose head is $x^*$.
    Let $z \in \epsilon_{r,B} (X)$ be arbitrary. Note that $z \neq x$ since $\InVertices_B(x) \subseteq X$ and there is no edge between $r$ and $x$.
    \begin{description}
    \item[If $a(z) \cap Y \neq \emptyset$] Since $z \in \epsilon_{r,B}(X)$, there exists an oriented edge $\ve \in E(B)$ such that its head is $z$, its tail is not contained in $X$ and there exists an $r$--$\ve$~path in $B$. Then $\ve$ in $a(B)$ has the property that its tail is not in $Y$, its head is in $a(z)$ and there exists an $r$--$\ve$~path in $a(B)$ by~\cref{cor:reach}.
    If the head of $\ve$ is in $Y$, $\ve \in \delta_{r,a(B)}^{\mathsf{path}} (Y ) \setminus F$ and we map $z$ to $\ve$.
    Otherwise, the orientation $\vf$ of the edge $z^+z^-$ whose tail is the head of $\ve$ is contained in $\delta_{r,a(B)}^{\mathsf{path}} (Y ) \setminus F$ and we map $z$ to $\vf$.
    \item[If $a(z) \cap Y = \emptyset$] Then $z \in \InVertices_B(x)$, which implies that there exists an oriented edge $\vg$ with tail  $z$ and head  $x$ for which there exists an $r$--$\vg$~path in $B$.
    Thus there is an $r$--$\vg$~path in $a(B)$ by~\cref{cor:reach}. Furthermore, the tail of $\vg$ is not in $Y$ and the head of $\vg$ is in $\mathbf{a}(x)$.
    Note that by the choice of $x^*$, the head of $\vg$ is $x^* \in Y$. Thus $\vg$ is in $F$ and we map $z$ to $\vg$.
    \end{description}
    \noindent
    Since $\Phi$ maps a vertex $z \in \epsilon_{r, B}(X)$ either to an edge of $F$ with tail in $a(z)$ or to an edge of $\delta_{r, a(B)}^{\textrm{path}}(Y) \setminus F$ with head in $a(z)$, $\Phi$ is indeed injective.
    
    Since $\kappa_B(r,x) = \lambda_{a(B)}^{\textrm{path}}(r,x^*) = |\delta_{r,a(B)}^{\mathsf{path}} (Y )|$, we can deduce $\kappa_B(r,x) \geq |\epsilon_{r,B} (X)|$. This completes the proof.
\end{proof}

\subsection{Vertex-flames}
We define \emph{vertex-flames} in bidirected graphs in a similar way as we defined flames:
An $r$-rooted bidirected graph $B$ is a \emph{vertex-flame} if the number of edges of $B$ is bounded by $\sum_{v \in V(D) \setminus \{r\}} \sum_{\alpha \in \{+,-\}} \kappa_B^\alpha(r,v)$.
We transfer~\cref{thm:edge_flame_bidirected} as follows:
    \begin{theorem} \label{thm:vertex_flame_bidirected}
    Every $r$-rooted clean bidirected graph $B$ admits an $r$-rooted vertex-flame $ F $ with $\kappa_F^\alpha(r,v)=\kappa_B^\alpha(r,v)$ for each $ v\in V(B) \setminus \{ r \} $ and $\alpha \in \{+,-\}$.
   \end{theorem}
\begin{proof}
    By~\cref{cor:reach}, the auxiliary graph $a(B)$ is edge-clean. Thus we can apply~\cref{thm:edge_flame_bidirected} to $a(B)$ and obtain a flame $\hat{F}$ with $\lambda_{\hat{F}}^{\alpha, \textrm{ path}}(r,v)=\lambda_{a(B)}^{\alpha, \textrm{ path}}(r,v)$ for each $ v\in V(a(B)) \setminus \{ r \} $ and $\alpha \in \{+,-\}$. Let ${F}$ be the subgraph of $B$ obtained from $\hat{F}$ by adding any auxiliary edges which may be missing, then contracting all auxiliary edges. We show that $F$ is the desired flame.

    We begin by showing that $\kappa_{{F}}^{\alpha}(r,v)=\kappa_B^{\alpha}(r,v)$ for each $ v\in V(B) \setminus \{ r \} $ and $\alpha \in \{+,-\}$. Let $v \in V(B) \setminus \{r\}$ and $\alpha \in \{+,-\}$ be arbitrary. We set $v^*:= v^\alpha$ if $v \in \Pi_B$ and we let $v^*:= v$ otherwise.
    By~\cref{cor:aux_connectivity} and by the definition of $\hat{F}$, there is a family $\mathcal{P}$ of $\kappa_B^\alpha(r,v)$ edge-disjoint $r$--$(v^*,\alpha)$~paths in $\hat{F}$.
    Then $\{g_B(P): P \in \mathcal{P} \}$ is a family of $\kappa_B^\alpha(r,v)$ internally vertex-disjoint $r$--$(v,\alpha)$~paths in $F$.

    It remains to show that $F$ is a vertex-flame.
    Note that $\Lambda_B:= \{v \in V(B) \setminus \{r\}: \kappa_B^+(r,v)=\kappa_B^-(r,v)=0\}$ is a subset of $\Pi_B$. Let $v \in V(B) \setminus \{r\}$ be arbitrary. By~\cref{cor:aux_connectivity} we can deduce the following:
    \begin{description}
        \item[If $v \in \Lambda_B$]
        \[
        \sum_{\beta \in \{+,-\}} \sum_{\alpha \in \{+,-\}} \lambda_{a(B)}^{\alpha, \textrm{ path}} (r,v^\beta)=0= \sum_{\alpha \in \{+,-\}} \kappa_B^\alpha(r,v).
        \]
        \item[If $v \in \Pi_B \setminus \Lambda_B$]
            There exists  a $\gamma \in \{+,-\}$ such that $\kappa_B^\gamma(r,v)=0$.
            Thus
            \[
            \sum_{\beta \in \{+,-\}} \sum_{\alpha \in \{+,-\}} \lambda_{a(B)}^{\alpha, \textrm{ path}} (r,v^\beta)= 1 + \kappa_B^{-\gamma}(r,v) = 1 + \sum_{\alpha \in \{+,-\}} \kappa_B^\alpha(r,v).
            \]
        \item[If $v \notin \Pi_B$] The vertex $v$ is also a vertex of $a(B)$ and thus:
        \[
            \sum_{\alpha \in \{+,-\}} \lambda_{a(B)}^{\alpha, \textrm{ path}} (r,v)= \sum_{\alpha \in \{+,-\}} \kappa_B^\alpha(r,v).
        \]
            
    \end{description}
    Thus by the choice of $\hat{F}$ we can deduce
    \[
    |E(\hat{F})| \leq \sum_{w \in V(a(B)) \setminus \{r\}} \sum_{\alpha \in \{+,-\}} \lambda_{a(B)}^{\alpha, \textrm{path}}(r,w) = |\Pi_B \setminus \Lambda_B| + \sum_{v \in V(B) \setminus \{r\}} \sum_{\alpha \in \{+,-\}} \kappa_B^\alpha(r,v).
    \]

    Finally, we argue that every auxiliary edge $v^+v^-$ for $v \in \Pi_B \setminus \Lambda_B$ is contained in $\hat{F}$.
    Let $v \in \Pi_B \setminus \Lambda_B$ be arbitrary. Then there is $\gamma \in \{+,-\}$ such that $\kappa_B^\gamma(r,v) \geq 1$. Thus $\lambda_B^{\gamma, \textrm{ path}}(r, v^{-\gamma}) = 1$ by~\cref{cor:aux_connectivity}. The auxiliary edge corresponding to $v$ is the unique edge of $a(B)$ ending in $v^{-\gamma}$ with sign $\gamma$ and thus has to be contained in $\hat{F}$.
    This implies
    \[
    |E(F)| \leq |E(\hat{F})| - |\Pi_B \setminus \Lambda_B| \leq \sum_{v \in V(B) \setminus \{r\}} \sum_{\alpha \in \{+,-\}} \kappa_B^\alpha(r,v)
    \]
    and completes the proof.
\end{proof}

\subsection{Pym's theorem for internally vertex-disjoint paths}
    \begin{theorem}\label{thm:pym-helper}
        Let $B$ be a clean bidirected graph rooted at $r$, let $x \in V(B) \setminus \{r\}$ and
        let $\mathcal{P}$ and $\mathcal{Q}$ be sets of internally vertex-disjoint $r$--$x$~paths in $B$. Then there is a set $\mathcal{R}$ of internally vertex-disjoint $r$--$x$ paths in $B$ such that $\mathcal{R}$ contains the first edges of $\mathcal{P}$ and the last edges of~$\mathcal{Q}$.
    \end{theorem}
    \noindent
    Note that \cref{thm:pym-helper} fails in general for non-clean bidirected graphs (see~\cref{fig:pym}).
    \begin{proof}
        If $x \in \Pi_B$, then there exists an $\alpha \in \{+,-\}$ such that $\kappa_B^{-\alpha}(r,x) = 0$, and we let $x^*:= x^\alpha$.
        Otherwise, set $x^*:= x$.
        Note that $g_B^{-1}$ maps every $r$--$x$~path in $B$ to an $r$--$x^*$~proper path in $a(B)$ by the choice of $x^*$.
        Thus $g_B^{-1}$ turns $\mathcal{P}$ and $\mathcal{Q}$ into sets $\Tilde{\mathcal{P}}$ and $\Tilde{\mathcal{Q}}$ of $r$--$x^*$~proper paths such that the first edges of $\mathcal{P}$ and $\Tilde{\mathcal{P}}$ are the same and the last edges of $\mathcal{Q}$ and $\Tilde{\mathcal{Q}}$ are the same.
        
        By~\cref{cor:reach}, the auxiliary graph $a(B)$ is edge-clean. Thus we can apply~\cref{thm:edge-pym-bidirected} to $\Tilde{\mathcal{P}}$ and $\Tilde{\mathcal{Q}}$ in $a(B)$ to obtain a set $\Tilde{\mathcal{R}}$ of proper $r$--$x^*$~paths such that $\Tilde{\mathcal{R}}$ contains the first edges of $\Tilde{\mathcal{P}}$ and the last edges of $\Tilde{\mathcal{Q}}$. Then $g_B$ maps $\Tilde{\mathcal{R}}$ to the desired set $\mathcal{R}$ of $r$--$x$~paths.
    \end{proof}

    \subsection{From internally vertex-disjoint to vertex-disjoint}
    Given two sets of vertices $X$ and $Y$, a path $P$ is an $\emph{$X$--$Y$~path}$ if the first vertex of $P$ is the unique vertex of $P$ in $X$ and the last vertex of $P$ is the unique vertex of $P$ in $Y$.
    \begin{corollary}\label{thm:set-menger}
    Let $B$ be a bidirected graph and let $X, Y$ be sets of vertices of $B$ such that there is no nontrivial path in $B$ with both endpoints in $X$.
    Then the maximum number of vertex-disjoint $X$--$Y$~paths in $B$ is equal to
    $
    \min \left\lbrace  \left| \epsilon_{X,B} (W)
            \right|\ :\ 
             Y \subseteq W \subseteq V(B) \right\rbrace
    $.
\end{corollary}
\noindent
Here, $\epsilon_{X,B} (W)$ is the set of vertices $v \in W$ for which there is an $X$--$\{v\}$~path that is either trivial or whose second to last vertex is not in $W$. 

    \begin{corollary}
	\label{thm:set-pym}
	Let $B$ be a bidirected graph, let $X$ and $Y$ be sets of vertices of $B$ such that there is no nontrivial path in $B$ with both endpoints in $X$, and let $\mathcal{P}, \mathcal{Q}$ be sets of vertex-disjoint $X$--$Y$~paths.
    Then there is a set $\mathcal{R}$ of vertex-disjoint $X$--$Y$ paths such that the set of first vertices of the paths in $\mathcal{P}$ is a subset of the set of first vertices of the paths in $\mathcal{R}$ and the set of last vertices of the paths in $\mathcal{Q}$ is a subset of the set of last vertices of the paths in $\mathcal{R}$.
\end{corollary}

    \begin{proof}[Proof of~\Cref{thm:set-menger,thm:set-pym}]
        We may assume that for every $x \in X$, all edges incident with $x$ have sign $+$ at $x$ since this does not affect what counts as an $X$--$Y$~path or as an $X$--$v$~path. For~\cref{thm:set-pym}, we assume further that all edges incident with $y \in Y$ have sign $+$ at $y$ since this does not affect $X$--$Y$~paths.
        
        Let $B'$ be the auxiliary graph constructed from $B$ in the following way. First, we add a vertex $r$ to $B$, and add for each $x \in X$ an edge $rx$ with sign $-$ at $x$ and sign $+$ at $r$. Second, we add a vertex $z$ to $B$, and for each vertex $y$ in $V(\Y)$, we add edges between $y$ and $z$ with signs $-$ and $+$ at $y$ and $z$ respectively.

        Since every nontrivial $r$--$r$~almost path in $B'$ contains a nontrivial $X$--$X$~path of $B$, the auxiliary graph $B'$ rooted at $r$ is clean.
        Furthermore, every $X$--$Y$~path extends to an $r$--$z$~path by adding edges incident with $r$ and $z$.
        Conversely, every $r$--$z$~path $P$ contains an $X$--$Y$~path. Additionally, if all edges incident with $y \in Y$ have the same sign at $y$, $P - r - z$ is an $X$--$Y$~path.
        Thus by applying~\Cref{thm:strong_menger_vertex,thm:pym-helper} to $B'$ and the vertices $r, z$ we obtain~\Cref{thm:set-menger,thm:set-pym}.
    \end{proof}

\section{The vertex-decomposition}\label{sec: vertex-decomp}
In this section we introduce another decomposition for rooted bidirected graphs which has similar properties for internally vertex-disjoint paths to those which the edge-decomposition has for edge-disjoint paths. 
The main advantage of the vertex-decomposition is that for these properties we only have to demand cleanness rather than edge-cleanness.
%

Recall that a bidirected graph $B$ rooted at $r$ is reachable if $B$ contains an $r$--$e$~almost path for every $e \in E(B)$ and an edge $e \in E(B)$ is directable if there is a unique orientation $\ve$ of $e$ (called the natural orientation) such that $B$ contains an $r$--$\ve$~almost path.
An edge $e \in E(B)$ is {\em undirectable} if $B$ contains both an $r$--$\ve$~almost path and an $r$--$\ev$~almost path.
The {\em undirectable components} of $B$ are the nontrivial connected components of the (underlying undirected) subgraph of $B$ induced by the set of undirectable edges.

Before we define the vertex-decomposition, we show that every undirectable component has a unique vertex in $\Pi_B \cup \{r\}$. 
\begin{lemma}\label{lem:comparing_components}
Let $B$ be a reachable bidirected graph rooted at $r$ and let $F \subseteq E(a(B))$.
Then $F$ is the edge set of an undirectable component of $B$ if and only if $F$ is the edge set of a trail-undirectable component of $a(B)$.
\end{lemma}

\begin{proof}
We begin by proving that none of the auxiliary edges of $a(B)$ are trail-undirectable:
Suppose for a contradiction that there is a trail-undirectable auxiliary edge $v^+v^-$ in $a(B)$ for some $v \in \Pi_B$.
Then there is an $r$--$\vec{v^+v^-}$~almost path and an $r$--$\cev{v^+v^-}$~almost path in $a(B)$ by \cref{lem:aux-graph-trail-to-almost path}.
Thus there is an $r$--$(v, +)$ and an $r$--$(v, -)$~path in $B$ by \cref{lem:aux_path_bijection}, contradicting that $v \in \Pi_B$.

By~\cref{cor:reach}, every edge $e \in E(B)$ is undirectable in $B$ if and only if $e$ is trail-undirectable in $a(B)$.
Thus the set of undirectable edges of $B$ is the same as the set of trail-undirectable edges of $a(B)$.

It now suffices to show that two undirectable edges are incident in $B$ if and only if they are incident in $a(B)$.
Let $e$ and $f$ be two undirectable edges that are incident in $B$, and let $c \in V(B)$ be a common endpoint of $e$ and $f$.
If $e$ and $f$ have the same sign at $c$, then they are incident in $a(B)$.
Otherwise, there is an $r$--$(c,+)$~path and an $r$--$(c,-)$~path in $B$ since $e$ and $f$ are undirectable in $B$.
Thus $c \notin \Pi_B$, which implies that $c$ is a vertex of $a(B)$ and thus common endpoint of $e$ and $f$ in $a(B)$.
Conversely, if two edges of $E(B)$ are incident in $a(B)$, they are incident in $B$.
\end{proof}

\begin{lemma}\label{lem:trail-solid-aux-graph}
    Let $B$ be a reachable bidirected graph rooted at $r$. Then $\{v^+, v^-: v \in \Pi_B\} \cup \{r\}$ is the set of trail-solid vertices of $a(B)$.
\end{lemma}
\begin{proof}
    The vertex $r$ is trail-solid by definition. We show that the vertices in $\{v^+, v^-: v \in \Pi_B\}$ are trail-solid. Let $v \in \Pi_B$ be arbitrary. By~\cref{lem:comparing_components}, the auxiliary edge $v^+v^-$ is directable. Up to symmetry, we assume without loss of generality that $\vec{v^+v^-}$ is the natural orientation of $v^+v^-$. Then $v^-$ is trail-solid by~\cref{cor:trail-solid}. Furthermore, for every edge $e \neq v^+v^-$ incident with $v^+$ there is no $r$--$\ve$~trail, where $\ve$ has tail $v$, since there is no $r$--$\cev{v^+v^-}$~trail. This implies that all edges incident with $v^+$ are directable and thus $v^+$ is trail-solid by~\cref{cor:trail-solid}.

    Now let $w \in V(a(B)) \setminus (\{v^+, v^-: v \in \Pi_B\} \cup \{r\}) = V(B) \setminus (\Pi_B \cup \{r\})$. Then $\kappa_B^\alpha(r,w) \geq 1$ for every $\alpha \in \{+,-\}$ by definition of $\Pi_B$. By~\cref{cor:aux_connectivity}, $\lambda_{a(B)}^\alpha(r,w) \geq 1$ for every $\alpha \in \{+,-\}$. Then~\cref{lem:trail-solid-one-sign} implies that $w$ is not trail-solid in $a(B)$.
\end{proof}

\begin{corollary}\label{cor:vtx-decomp-unique-solid}
    Every undirectable component of a rooted reachable bidirected graph contains a unique vertex in $\Pi_B \cup \{r\}$.
\end{corollary}
\begin{proof}
Let $C$ be some undirectable component of $B$. By~\cref{lem:comparing_components}, there exists a trail-undirectable component $C'$ of $a(B)$ with $E(C)=E(C')$.
Note that for every $v \in V(B)$ we have $v \in V(C)$ if and only if $\mathbf{a}(v) \cap V(C') \neq \emptyset$.
By~\cref{lem:trail-solid-aux-graph}, for every vertex $v \in V(B) \setminus \{r\}$ we have $v \in \Pi_B \cup \{r\}$ if some $w \in \mathbf{a}(v)$ is trail-solid, moreover, every $w \in \mathbf{a}(v)$ is trail-solid if $v \in \Pi_B \cup \{r\}$.
We can deduce that there exists a unique vertex of $C$ in $\Pi_B \cup \{r\}$, since there is a unique trail-solid vertex of $C'$.
\end{proof}

Let $B$ be a reachable bidirected graph rooted at $r$ and let $C_1, \hdots, C_k$ be its undirectable components. For each $i \in [k]$, let $c_i$ be the unique vertex of $C_i$ in $\Pi_B \cup \{r\}$. Let $\hat{B}$ be the bidirected graph formed from $B$ by contracting component $C_i$ into vertex $c_i$ for $i = 1, \hdots, k$. 
Note that $E(\hat{B})$ is exactly the set of directable edges of $B$. (However, we again stress that $\hat{B}$ is not a substructure of $B$ in the intuitive sense. Specifically, two edges $e$ and $f$ that are incident in $\hat{B}$ at vertex $c_i$ for $i = 1, \hdots, k$ need not be incident in $B$.)
Let $D_B^{\mathrm{path}}$ be the rooted directed graph formed from $\hat{B}$ by giving every edge its natural orientation; call $D_B^{\mathrm{path}}$ the {\em skeleton} of $B$. See~\cref{fig:vertex-decomp} for an example and compare with the trail-skeleton in~\cref{fig:edge-decomp}.

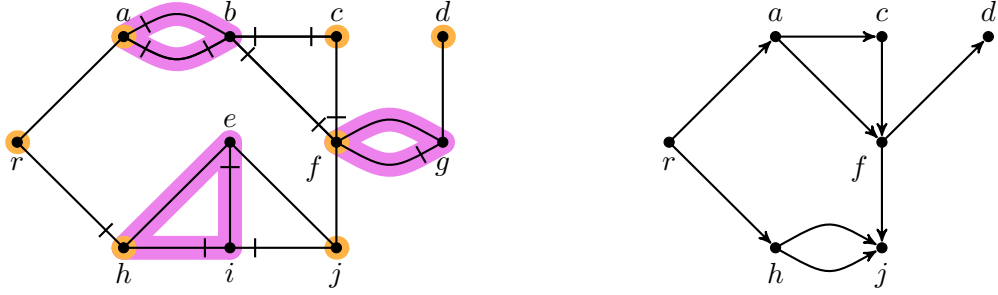
\begin{figure}
\centering
\begin{subfigure}[b]{0.54\textwidth}
\centering
\begin{tikzpicture}
\coordinate (r) at (0,1.4);
\coordinate (a) at (1.4,2.8);
\coordinate (b) at (2.8,2.8);
\coordinate (c) at (4.2,2.8);
\coordinate (d) at (5.6,2.8);
\coordinate (e) at (2.8,1.4);
\coordinate (f) at (4.2,1.4);
\coordinate (g) at (5.6,1.4);
\coordinate (h) at (1.4,0);
\coordinate (i) at (2.8,0);
\coordinate (j) at (4.2,0);

\foreach \i/\j in {h/e,e/i,h/i}{
    \draw[line width=9pt, LavenderMagenta, line cap=round] (\i) to (\j);
}

\foreach \i/\j in {a/b,f/g}{
    \draw [line width=9pt, LavenderMagenta, line cap=round] (\i) ..controls +($-0.5*(\i)+0.5*(\j)+(0,+0.4)$) .. (\j);
    \draw [line width=9pt, LavenderMagenta, line cap=round] (\i) ..controls +($-0.5*(\i)+0.5*(\j)+(0,-0.4)$) .. (\j);
}
\foreach \i in {r,a,c,d,f,h,j}{
    \draw[line width=9pt, PastelOrange, line cap=round] (\i) to (\i);
}

\foreach \i in {a,b,c,d,e}{
    \node[dot] at (\i) [label=above:$\i$] {};
}
\foreach \i in {f}{
    \node[dot] at (\i) [label=below left:$\i$] {};
}
\foreach \i in {h,i,j,g,r}{
    \node[dot] at (\i) [label=below:$\i$] {};
}
\foreach \i/\j in {r/a, h/e, e/j, f/j, d/g}{
    \draw[thick] (\i) to (\j);
}
\foreach \i/\j in {r/h, h/i, i/e, j/i, c/f, b/c, c/b, b/f, f/b}{
    \draw[-e+8,thick] (\i) to (\j);
}
\foreach \i/\j in {a/b, b/a, f/g}{
     \draw [-e+8,thick] (\i) ..controls +($-0.5*(\i)+0.5*(\j)+(0,-0.4)$) .. (\j);
}
\foreach \i/\j in {b/a}{
    \draw [-e+8,thick] (\i) ..controls +($-0.5*(\i)+0.5*(\j)+(0,+0.4)$) .. (\j);
}
\foreach \i/\j in {f/g}{
    \draw [thick] (\i) ..controls +($-0.5*(\i)+0.5*(\j)+(0,+0.4)$) .. (\j);
}

\end{tikzpicture}
\caption{A bidirected graph $B$ rooted in $r$. The vertices of $\Pi_B \cup \{r\}$ are marked orange and the undirectable edges of $B$ are marked purple.}
\end{subfigure}
\hfill
\begin{subfigure}[b]{0.45\textwidth}
\centering
\begin{tikzpicture}
\coordinate (r) at (0,1.4);
\coordinate (a) at (1.4,2.8);
\coordinate (c) at (2.8,2.8);
\coordinate (d) at (4.2,2.8);
\coordinate (f) at (2.8,1.4);
\coordinate (h) at (1.4,0);
\coordinate (j) at (2.8,0);

\foreach \i in {a,c,d}{
    \node[dot] at (\i) [label=above:$\i$] {};
}
\foreach \i in {r,h,j}{
    \node[dot] at (\i) [label=below:$\i$] {};
}
\foreach \i in {f}{
    \node[dot] at (\i) [label=below left:$\i$] {};
}
\foreach \i/\j in {r/a,a/c,a/f/c/f,f/d,r/h,f/j,c/f}{
    \draw[arc,thick] (\i) to (\j);
}
\draw [arc,thick] (h) ..controls +($-0.5*(h)+0.5*(j)+(0,+0.4)$) .. (j);
\draw [arc,thick] (h) ..controls +($-0.5*(h)+0.5*(j)+(0,-0.4)$) .. (j);
\end{tikzpicture}
\caption{The skeleton of $B$. \newline \newline}
\end{subfigure}
\caption{A rooted bidirected graph and its skeleton.}
    \label{fig:vertex-decomp}
\end{figure}

\begin{lemma}\label{lem:vertices-of-skeleton}
    For every rooted reachable bidirected graph $B$ we have $V(\skeleton{B}) = \Pi_B \cup \{r\}$.
\end{lemma}
\begin{proof}
    By the construction of $\skeleton{B}$, $\Pi_B \cup \{r\} \subseteq V(\skeleton{B})$. Suppose for a contradiction that there is a $v \in V(\skeleton{B}) \setminus (\Pi_B \cup \{r\})$. Then there is an $r$--$(v,+)$~path $P$ and an $r$--$(v,-)$~path $Q$ in $B$. Let $e$ be the terminal edge of $P$. Then the almost paths $P$ and $Q + e$ witness that $e$ is undirectable. This implies that $v$ is contained in some undirectable component $C$ of $B$. Since $v \in \skeleton{B}$ the construction of $\skeleton{B}$ ensured that $v \in (\Pi_B \cup \{r\})$, a contradiction.
\end{proof}

    Now we investigate the correspondence between the skeleton and the trail-skeleton.
    Changing a bidirected graph by switching all signs at some vertex $v$, i.e.\ turning all $+$~signs at $v$ into $-$~signs and simultaneously turning all $-$~signs at $v$ into $+$~signs, is called a \emph{sign-switch}.
    We write $B' \equiv B$, if a bidirected graph $B'$ is obtained from a bidirected graph $B$ by a sequence of sign-switches.
    Note that sign-switches do not affect what counts as a path.
\begin{lemma}\label{lem:path_skeleton_to_trail_skeleton}
    For every rooted reachable bidirected graph $B'$ there exists a bidirected graph $B$ with $B' \equiv B$ and $a(\skeleton{B}) = \trailskeleton{a(B)}$.
\end{lemma}
\noindent
Here, we define $a(\skeleton{B})$ similarly as in~\cref{sec:aux-graph}:
We interpret $\skeleton{B}$ as a bidirected graph, i.e.\ giving heads of edges sign~$+$ and tails of edges sign~$-$. Then we apply the auxiliary graph from~\cref{sec:aux-graph} to $\skeleton{B}$ and let $a(\skeleton{B})$ be the directed graph obtained from it by turning $+$~signs into heads and $-$~signs into tails.
\begin{proof}
    By~\cref{lem:comparing_components}, every auxiliary edge of $a(B')$ is trail-directable. Let $B$ be the bidirected graph with $B \equiv B'$ such that for every $v \in \Pi_B$ the natural orientation of $v^+v^-$ in $a(B)$ is $\vec{v^+v^-}$.
    The choice of $B$ implies that every non-auxiliary edge incident with $v^+$ is trail-directable and has head $v^+$ in $a(B)$, and there is no trail-directable non-auxiliary edge that has head $v^-$ in $a(B)$.

    Since $\skeleton{B}$ is a directed graph, every vertex of $V(\skeleton{B}) \setminus \{r\}$ is plain.
     By~\Cref{lem:vertices-of-skeleton,lem:trail-solid-aux-graph}, we have
     \[
    V(a(\skeleton{B})) = \{v^+, v^-: v \in V(\skeleton{B}) \setminus \{r\}\} \cup \{r\} = \{v^+, v^-: v \in \Pi_B\} \cup \{r\} = V(\trailskeleton{a(B)}).
     \]

    We note that the set $F:=\{\vec{v^+v^-}: v \in \Pi_B\}$ is included both in $E(a(\skeleton{B}))$ and $E(\trailskeleton{a(B)})$ by the definition of the auxiliary graph and the choice of $B$.
    Next we show that the set of directable edges of $B$ corresponds one-to-one to $E(a(\skeleton{B})) \setminus F$ such that the natural orientation in $B$ and the orientation in $a(\skeleton{B})$ coincide.
    By definition of $\skeleton{B}$, $\ve$ is the natural orientation of a directable edge of $B$ if and only if $\ve$ is an edge of $\skeleton{B}$. Furthermore, $\ve$ is an edge of $\skeleton{B}$ if and only if $\ve$ is in $E(a(\skeleton{B})) \setminus F$.

    Moreover, the set of directable edges of $B$ corresponds one-to-one to $E(\trailskeleton{a(B)}) \setminus F$ such that the natural orientation in $B$ and the orientation in $\trailskeleton{a(B)}$ coincide. By~\cref{cor:reach}, $\ve$ is the natural orientation of a directable edge in $B$ if and only if $\ve$ is the natural orientation of a trail-directable edge in $E(a(B)) \setminus \{v^+v^-: v \in \Pi_B\}$.
    Furthermore, $\ve$ is the natural orientation of a trail-directable edge in $E(a(B)) \setminus \{v^+v^-: v \in \Pi_B\}$ if and only if $\ve$ is contained in $E(\trailskeleton{a(B)}) \setminus F$.

    Thus it suffices to show that for every natural orientation of a directable edge of $B$ the corresponding heads and tails in $a(\skeleton{B})$ and $\trailskeleton{a(B)}$ coincide. Let $\ve$ be an arbitrary natural orientation of a directable edge of $B$, and let $\rho \in \{\textrm{head, tail}\}$ be arbitrary. Furthermore, let $a$ be $\rho$ of $\ve$ in $B$.

    \begin{description}
        \item[If $a = r$] Then $r$ is the $\rho$ of $\ve$ in both $a(\skeleton{B})$ and $\trailskeleton{a(B)}$.
        \item[If $a \in \Pi_B$] Then $a$ is the $\rho$ of $\ve$ in $\skeleton{B}$. Thus the $\rho$ of $\ve$ in $a(\skeleton{B})$ is $a^+$ if $\rho = \textrm{head}$ and $a^-$ if $\rho = \textrm{tail}$ by the definition of the auxiliary graph. Furthermore, the $\rho$ of $\ve$ in $a(B)$ is $a^+$ if $\rho = \textrm{head}$ and $a^-$ if $\rho = \textrm{tail}$ by the choice of $B$. Note that the $\rho$ of $\ve$ coincides in $a(B)$ and $\trailskeleton{a(B)}$.
        \item[If $a \notin \Pi_B \cup \{r\}$] Then $a$ is the $\rho$ of $\ve$ in $a(B)$ and $a$ is not trail-solid in $a(B)$ by~\cref{lem:trail-solid-aux-graph}.  Thus $\ve$ witnesses that $\rho = \textrm{tail}$ by~\cref{cor:trail-solid}. On the one hand, there exists an undirectable component $C$ in $B$ that contains $a$. The tail of $\ve$ in $\skeleton{B}$ is the unique plain vertex $w$ of $C$. Then the tail of $\ve$ in $a(\skeleton{B})$ is $w^-$.
        On the other hand, there exists a trail-undirectable component $C'$ in $a(B)$ with $E(C') = E(C)$, by~\cref{lem:comparing_components}, and thus $C'$ contains $a$. Note that $w^-$ and $w^+$ are trail solid by~\cref{lem:trail-solid-aux-graph}, and at least one of them is contained in $C'$. Note that $w^-$ is contained in $C'$ since no edge incident with $w^+$ is trail undirectable by the choice of $B$. Thus $w^-$ is the unique trail-solid vertex of $C'$ and therefore $w^-$ is the tail of $\ve$ in $\trailskeleton{a(B)}$. \qedhere
    \end{description}
\end{proof}

By~\cref{lem:path_skeleton_to_trail_skeleton}, one can observe that the skeleton of a clean bidirected graph has similar properties for internally vertex-disjoint paths to those which the trail-skeleton of an edge-clean bidirected graph has with respect to edge-disjoint paths. As an example, we show:

\begin{theorem}
Let $B$ be a clean, path-reachable bidirected graph rooted at $r$.
Then
\begin{itemize}
    \item $\kappa_B(r,x) = \kappa_{\skeleton{B}}(r,x)$ for all $x \in V(\skeleton{B}) \setminus \{r\}$; and
    \item $\kappa_B(r,x) = 1$ for all $x \in V(B) \setminus V(\skeleton{B})$.
\end{itemize}
\end{theorem}
\begin{proof}
    We can assume without loss of generality that $B$ is as in~\cref{lem:path_skeleton_to_trail_skeleton}.
    First, let $x \in V(\skeleton{B}) \setminus \{r\}$ be arbitrary. Then $\kappa_B(r,x) = \max_{\alpha \in \{+,-\}}\lambda_{a(B)}^{\textrm{path}}(r,x^\alpha)$ by~\cref{cor:aux_connectivity}.
    By~\cref{cor:path-reachable}, $a(B)$ is path-reachable and thus \Cref{lem:disjoint_paths_skeleton,lem:disjoint_paths_skeleton_inside} imply $\lambda_{a(B)}^{\textrm{path}}(r,x^\alpha)= \lambda_{\trailskeleton{a(B)}}(r,x^\alpha)$.
    Furthermore, by~\cref{lem:path_skeleton_to_trail_skeleton}, we have $\lambda_{\trailskeleton{a(B)}}(r,x^\alpha) = \lambda_{a(\skeleton{B})}(r,x^\alpha)$.
    Finally, $\max_{\alpha \in \{+,-\}}\lambda_{a(\skeleton{B})}(r,x^\alpha) = \kappa_{\skeleton{B}}(r,x)$ by~\cref{cor:aux_connectivity}.

    Second, let $x \in V(B) \setminus V(\skeleton{B})$ be arbitrary. Then $x \notin \Pi_B \cup \{r\}$ by~\cref{lem:vertices-of-skeleton}.
    Since $B$ is clean, by~\cref{lem:non_plain_vertex}, we have $\kappa_B(r,x) = 1$.
\end{proof}

Finally, we note that, just as in the proof of~\cref{lem:edge-clean_characterisation}, one can observe that a bidirected graph $B$ is clean if and only if there does not exist an undirectable component containing $r$.

\section{Open questions}

We conclude with two open questions that ask about possible extensions of our work. 

First, given the relationship between bidirected graphs and graphs with perfect matchings, it seems natural to ask whether nice analogs of our structure theorems hold in graphs with perfect matchings. Our structure theorems have the following form: we model a bidirected graph using a directed graph whose structure closely resembles the structure of the bidirected graph. Since the correspondence between bidirected graphs and graphs with perfect matchings sends directed graphs to bipartite graphs with perfect matchings, a graph analog could involve modeling a graph with a perfect matching using a bipartite graph with a perfect matching whose structure closely resembles the original graph. 

We might hope that something even more specific holds: that in a highly connected graph with a perfect matching, we can find a highly connected bipartite matching minor. This would be a step towards the general case of Norin's Matching Grid Conjecture. Our paper shows that a similar result holds in bidirected graphs (under the clean and rooted assumptions), giving evidence and a possible roadmap for such a result.


\section*{Acknowledgements}
We thank Ebrahim Ghorbani and Raphael W. Jacobs for valuable discussions about the structure of bidirected graphs.

\bibliography{reference}

\end{document}